\documentclass[twoside,reqno]{amsart}
\usepackage{amsmath,amsthm,amssymb}
\usepackage{color,pdfsync}
\usepackage[frenchb,english]{babel}
\theoremstyle{plain}%
\newtheorem{theorem}{Theorem}[]%
\newtheorem{proposition}[theorem]{Proposition}%
\newtheorem{lemma}[theorem]{Lemma}%
\newtheorem{remark}[theorem]{Remark}%
\newtheorem{definition}[theorem]{Definition}%
\definecolor{colorGreen}{rgb}{0.,.67,0}
\definecolor{colorRed}{rgb}{0.67,0.,0}
\definecolor{colorBlue}{rgb}{0.,0.,0.67}

\begin{document}
%Partial Differential Equations/Equations aux D\'eriv\'ees Partielles
%
%\begin{frontmatter}
%% use the tnoteref command within \title for footnotes;
%% use the tnotetext command for theassociated footnote;
%% use the fnref command within \author or \address for footnotes;
%% use the fntext command for theassociated footnote;
%% use the corref command within \author for corresponding author footnotes;
%% use the cortext command for theassociated footnote;
%% use the ead command for the email address,
%% and the form \ead[url] for the home page:
%% \title{Title\tnoteref{label1}}
%% \tnotetext[label1]{}
%% \author{Name\corref{cor1}\fnref{label2}}
%% \ead{email address}
%% \ead[url]{home page}
%% \fntext[label2]{}
%% \cortext[cor1]{}
%% \address{Address\fnref{label3}}
%% \fntext[label3]{}

%\selectlanguage{francais}
%\title{French title}

%\selectlanguage{english}
\title[Thin film model with capillary effects and insoluble surfactant]{%
Weak solutions to a thin film model with\\ capillary effects and insoluble surfactant
}%
\thanks{This work was partially supported by the french-german PROCOPE project 20190SE}
%
%
%\author[addrHannover]{Joachim Escher}
%\ead{escher@ifam.uni-hannover.de}
%\author[addrDauphine]{Matthieu Hillairet}
%\ead{hillairet@ceremade.dauphine.fr}
%\author[addrToulouse]{Philippe Lauren\c{c}ot}
%\ead{laurenco@math.univ-toulouse.fr}
%\author[addrHannover]{Christoph Walker}
%\ead{walker@ifam.uni-hannover.de}
%
\author{Joachim Escher}
\address{Leibniz Universit\"at Hannover, Institut f\" ur Angewandte Mathematik \\ Welfengarten 1, D--30167 Hannover, Germany}
\email{escher@ifam.uni-hannover.de}
\author{Matthieu Hillairet}
\address{CEREMADE, UMR CNRS~7534, Universit\'e de Paris-Dauphine \\ Place du Mar\'echal De Lattre De Tassigny,
F--75775 Paris Cedex 16, France}
\email{hillairet@ceremade.dauphine.fr}
\author{Philippe Lauren\c{c}ot}
\address{Institut de Math\'ematiques de Toulouse, CNRS UMR~5219, Universit\'e de Toulouse \\ F--31062 Toulouse Cedex 9, France}
\email{laurenco@math.univ-toulouse.fr}
\author{Christoph Walker}
\address{Leibniz Universit\"at Hannover, Institut f\" ur Angewandte Mathematik \\ Welfengarten 1, D--30167 Hannover, Germany}
\email{walker@ifam.uni-hannover.de}
%
%
%\address[addrHannover]{Leibniz Universit\"at Hannover, Institut f\" ur Angewandte Mathematik, \\ Welfengarten 1, D--30167 Hannover, Germany}
%\address[addrDauphine]{CEREMADE, UMR CNRS~7534, Universit\'e de Paris-Dauphine, \\ Place du Mar\'echal De Lattre De Tassigny, F--75775 Paris Cedex 16, France}
%\address[addrToulouse]{Institut de Math\'ematiques de Toulouse, CNRS UMR~5219, Universit\'e de Toulouse, \\ F--31062 Toulouse Cedex 9, France}
%
%%%%%%%%%%%%%%%%%%%%%%%%%%%%%%%%%%%%%%%%%%%%%%%%%%%%%%%%%%%%%%%%
%
\begin{abstract}
The paper focuses on a model describing the spreading of an insoluble surfactant on a thin viscous film with capillary effects taken into account. The governing equation for the film height is degenerate parabolic of fourth order and coupled to a { second order} parabolic equation for the surfactant concentration. It is shown that nonnegative weak solutions exist under { natural} assumptions on the surface tension coefficient.
\end{abstract}
%
%%%%%%%%%%%%%%%%%%%%%%%%%%%%%%%%%%%%%%%%%%%%%%%%%%%%%%%%%%%%%%%%
%
%\begin{keyword}
%% keywords here, in the form: keyword \sep keyword
%% PACS codes here, in the form: \PACS code \sep code
%% MSC codes here, in the form: \MSC code \sep code
%% or \MSC[2008] code \sep code (2000 is the default)
%\end{keyword}
%
%
%\end{frontmatter}
%
%
\selectlanguage{english}
%\section*{Version fran\c{c}aise abr\'eg\'ee\footnote{Solutions faibles d'un mod\`ele de film mince recouvert d'un agent tensioactif insoluble}}
%
%\linenumbers

%
%\begin{theoreme}
%\end{theoreme}
%
%
\maketitle

\selectlanguage{english}
%
%
%--------------------------------------------------------------------
%
\section{Introduction}\label{sec:intro}
%
%--------------------------------------------------------------------
%
The modeling of the spreading of an insoluble surfactant on a thin viscous film leads to a coupled system of degenerate parabolic equations describing the space and time evolution of the height $h\ge 0$ of the film and the surface concentration $\Gamma\ge 0$ of surfactant. Assuming the film thickness to be small enough so that lubrication theory is applicable, taking into account capillary effects { but} neglecting gravitational and intermolecular (van der Waals) forces, the following system is obtained \cite{dWGC94,GW06,JG92} 
\begin{eqnarray}
\label{a1}
\partial_t h + \partial_x \left( \frac{1}{3}\ h^3\ \partial_x^3 h + \frac{1}{2}\ h^2\ \partial_x \sigma(\Gamma) \right) & =& 0\,, \quad (t,x)\in (0,\infty)\times (0,1)\,, \\
\label{a2}
 \partial_t \Gamma + \partial_x \left( \frac{1}{2}\ h^2\ \Gamma\ \partial_x^3 h + h\ \Gamma\ \partial_x \sigma(\Gamma) \right) & = & D\ \partial_x^2 \Gamma\,, \quad (t,x)\in (0,\infty)\times (0,1)\,, 
\end{eqnarray}
with homogeneous Neumann boundary conditions
\begin{equation}
\label{a3}
\partial_x h(t,x) = \partial_x^3 h(t,x) = \partial_x \Gamma(t,x) = 0\,, \quad (t,x)\in (0,\infty)\times \{0,1\}\,,
\end{equation}
and initial conditions
\begin{equation}
\label{a4}
(h,\Gamma)(0)=(h_0,\Gamma_0)\,, \quad x\in (0,1)\,.
\end{equation}
Here, $\sigma(\Gamma)$ denotes the surface tension which depends on the local concentration of surfactant, and $D>0$ stands for the surface diffusivity of the surfactant. Since $\sigma$ is defined up to a constant, we may assume without loss of generality that $\sigma(1)=0$ as a normalization condition. As the presence of surfactant reduces surface tension, $\sigma$ is a non-increasing function of $\Gamma$; for instance, 
\begin{equation}
\label{a5}
\sigma_\beta(s) := (\beta+1)\ \left[ 1- s + \left( \frac{\beta+1}{\beta} \right)^{1/3}\ s \right]^{-3} - \beta\,, \quad s\ge 0\,,
\end{equation}
with $\beta\in (0,\infty)$ and its limit as $\beta\to\infty$ (which is often assumed in applications)
\begin{equation}
\label{a6}
\sigma_\infty(s) := 1-s\,, \quad s\ge 0\,.
\end{equation}
The system \eqref{a1}-\eqref{a4} is a fully coupled nonlinear system of parabolic equations featuring a degeneracy where $h$ vanishes, a fact which cannot be excluded \textit{a priori}. Thus, classical solutions are unlikely to exist for all times in general and only local existence of smooth solutions to \eqref{a1}--\eqref{a2} in the absence of capillarity have been shown in \cite{Re96b,Re97}. The alternative is to study the Cauchy problem in a { framework} of weak solutions (see Section~\ref{sec:mr} for a precise definition) and this approach has been successfully employed to establish the existence of weak solutions to systems similar to \eqref{a1}--\eqref{a2}, for instance when the capillarity effect are neglected but the gravitational ones are accounted for \cite{EHLWxx} or when $\Gamma$ is replaced by $\lambda(\Gamma)=\max{\{0, 1-\Gamma\}}+1$ in \eqref{a2} \cite{BGN03,BN04}, a finite element numerical scheme being also developed in these two papers.

To our knowledge, existence of weak solutions has only been tackled in \cite{GW06} where { this} is proved under further technical assumptions on the surface tension $\sigma$. Namely, given a surface tension $\sigma \in \mathcal{C}^{2,1}_{loc}(\mathbb R),$  defining the free energy $g_{\sigma}$ by
\begin{equation} \label{a8}
g_{\sigma}(1) = g_{\sigma}'(1)=0,  \qquad  g_{\sigma}''(s) = - \dfrac{\sigma'(s)}{s}\qquad \text{ for } \, s \in \mathbb R,
\end{equation}
the authors assume the following:
\begin{enumerate}
\item[(A4)] The function $g_{\sigma}$ lies in $\mathcal{C}^{2,1}_{loc}(\mathbb R).$
\item[(A5)] There exists $c_g >0$ such that $g''_{\sigma}(s) \geq c_g$ for all $s \in \mathbb R.$
\item[(A6)] There exist $C_g$ { and some $r \in (0,2)$} for which $g''_{\sigma}(s) \leq C_g(|s|^r +1)$ for all $s \in \mathbb R$.
\end{enumerate}
The need in \cite{GW06} to define $\sigma$ and $g_\sigma$ in $\mathbb{R}$ instead of the physically relevant range $[0,\infty)$ for surfactant concentration stems from the fact that the weak solution $(h,\Gamma)$ to \eqref{a1}-\eqref{a2} constructed in \cite{GW06} might not satisfy $\Gamma\ge 0$. The extension of $\sigma$ and $g_\sigma$ to negative values then induces several limiting conditions. Namely, the convexity (A5) of $g_\sigma$ in $\mathbb{R}$ requires not only that $\sigma'(s)<0$ for $s>0$ as expected but also $\sigma'(s)>0$ for $s<0$, which implies $\sigma'(0)=0$ and thus excludes surface tensions $\sigma$ like $\sigma_\beta$ in \eqref{a5} and $\sigma_\infty$ in \eqref{a6}. In addition, assumption (A5) yields $\sigma'(s)\le -c_g \ s$ for $s\in \mathbb R$ so that $\sigma$ necessarily has a quadratic decay at infinity. This again excludes the previous examples $\sigma_\beta$ in \eqref{a5} and $\sigma_\infty$ in \eqref{a6}.

%The free-energy { $g_\sigma$} is also assumed  ``genuinely" positive. We note that, in \eqref{a1}--\eqref{a2}, the surface tension $\sigma$ is defined up to a constant so that we { may} fix $\sigma(1) = 0$ as a normalization condition. We assume that surface tensions and free energies are normalized this way throughout the paper. Then, a sufficient condition for $g_{\sigma}$ to be positive is that $\sigma'(s) \geq 0$ for  $s <0$ and $\sigma'(s) \leq 0$ for positive $s>0$. { The approach in \cite{GW06} does not guarantee a nonnegative surfactant concentration $\Gamma$. Consequently, the construction therein requires a non-physical extension of $\sigma$ to negative values.} The  assumptions (A4)--(A5) { induce several limiting conditions}. Firstly, assumption (A5) { yields} $\sigma'(s)\le -c_g \ s$ for $s\in \mathbb R$ so that $\sigma$ { necessarily has a quadratic decay} at infinity. Secondly, assumption (A4) requires that $\sigma'$ vanishes linearly close to $0.$ Neither $\sigma_{\beta}$ (see \eqref{a5}) nor $\sigma_{\infty}$ (see \eqref{a6}) satisfy both of these assumptions.\\

The aim of the present paper is to construct a weak solution to \eqref{a1}--\eqref{a2} under weaker assumptions on the surface tension $\sigma$ (satisfied in particular by $\sigma_\infty$) and such that both $h$ and $\Gamma$ are nonnegative throughout time evolution. More precisely, we assume that the surface tension $\sigma$ satisfies:
%{ In the present paper we prove} existence of at least one { nonnegative} solution to \eqref{a1}-\eqref{a2} { for} a surface tension $\sigma$ satisfying
\begin{itemize}
\item[(H1)] $\sigma \in \mathcal{C}^{1}((0,\infty)) \cap \mathcal{C}([0,\infty))$ with $\sigma(1)=0$. 
\item[(H2)] There exist $\sigma_0,{ \sigma_{1}} \in (0,\infty)$ and $\theta \in [0,1)$ for which:
\begin{equation} \label{case.2}
-{\sigma}_0 < \sigma'(s) \leq -\dfrac{{ \sigma_{1}}}{1+ s^{\theta}} \quad \text{ for } \, s  \ge 1, \qquad 
-\sigma_{0}< \sigma'(s) < 0 \quad \text{ for } \, s  \in (0,1).
\end{equation}
\end{itemize}
The assumptions (H1)--(H2) include physically relevant surface tensions $\sigma$, which may slowly decrease at infinity.  
In particular, $\sigma_\infty$ is included (but not $\sigma_\beta$ for $\beta\in (0,\infty)$).\\

In the next section, we introduce the definition of weak solutions and give precise statements for
our existence result. To prove this result, we first construct {\it nonnegative} solutions in the framework of \cite{GW06} under assumptions (A4)--(A6). { This improves the results of \cite{GW06} in that the surfactant concentration $\Gamma$ stays nonnegative through time evolution}. This is the content of {\bf Section~\ref{sec:ens}}. Then we extend the construction to a surface tension $\sigma$ merely satisfying (H1)--(H2) by approximating $\sigma$ with surface tensions $\sigma_k$ satisfying (A4)--(A6) and studying compactness properties of their associated weak solutions. The construction of $\sigma_k$ and the compactness argument are presented in {\bf Section~\ref{sec:stabsigma}}.

%--------------------------------------------------------------------
%
\section{Weak solutions and main results} \label{sec:mr}
%
%--------------------------------------------------------------------

To introduce the definition of weak solutions, we first derive energy estimates satisfied by smooth nonnegative
solutions to \eqref{a1}--\eqref{a2}. So, let us consider a non-increasing and smooth surface tension $\sigma$ and a smooth solution $(h,\Gamma)$ to \eqref{a1}--\eqref{a3} in $(0,T)\times (0,1)$ for some $T>0$, both functions being uniformly bounded from below by a positive constant.
%we assume that { $h,\Gamma \in \mathcal{C}^{\infty}([0,T] \times [0,1])$ with $T>0$ are solutions to \eqref{a1}--\eqref{a2} being} uniformly bounded { from} below by a { positive} constant and { satisfying} \eqref{a1}--\eqref{a3} with { a} non-increasing surface tension $\sigma \in \mathcal{C}^{\infty}([0,\infty)).$
First, we note that by \eqref{a1}-\eqref{a3} there holds 
\begin{equation} \label{aconserver}
\dfrac{\textrm{d}}{\textrm{dt}} \left[ \displaystyle{\int_0^1} h \text{ d$x$} \right]  = 0, \qquad \dfrac{\textrm{d}}{\textrm{dt}} \left[ \displaystyle{\int_0^1} \Gamma \text{ d$x$} \right]  = 0.
\end{equation}
Setting
\begin{equation} \label{a8bis}
g_{\sigma}(1) = g_{\sigma}'(1)=0,  \qquad  g_{\sigma}''(s) = - \dfrac{\sigma'(s)}{s}\qquad \text{for} \, s \in (0,\infty), 
\end{equation}
 it follows from \eqref{a1}--\eqref{a3} that
\begin{eqnarray*} 
& & \dfrac{\textrm{d}}{\textrm{dt}} \left[ \displaystyle{\int_{0}^1} \left( \dfrac{|\partial_x h|^2}{2} + g_{\sigma}(\Gamma)\right) \text{ d$x$} \right] = \displaystyle{\int_{0}^1} \left[  \partial^2_{x} h \; \partial_x \left( \dfrac{h^{3}}{3} \ \partial^3_{x} h + \dfrac{h^2}{2} \partial_x \sigma(\Gamma) \right) \right] \text{ d$x$}\\[4pt]
& & \hspace{3cm} +  \displaystyle{\int_{0}^1} \left[ g_{\sigma}''(\Gamma) \; \partial_x \Gamma \; \left(\dfrac{h^{2}}{2} \ \Gamma \ \partial^3_{x} h + {h} \ \Gamma \ \partial_x \sigma(\Gamma) -D \ \partial_x \Gamma \right)  \right] \text{ d$x$} \\
\\[4pt]
& & \quad = - \displaystyle{\int_{0}^1} \left[  \dfrac{h^3}{3} \ |\partial^3_{x} h|^2 + h^2\ \partial_x \sigma(\Gamma) \ \partial^3_{x} h+ h |\partial_x \sigma(\Gamma)|^2 - D \dfrac{\sigma'(\Gamma)}{\Gamma} |\partial_x \Gamma|^2  \right] \text{ d$x$} .
\phantom{12345678}
\end{eqnarray*}
We { abbreviate} product terms { by} introducing $J_s^2$ and $J_f^2$, where
\begin{eqnarray}
\label{a10}
J_f = J_f[h,\Gamma] & :=&  \frac{h^{3/2}}{3} \ \partial^3_{x} h  + \frac{{h^{1/2}}}{2} \ \partial_{x} \sigma(\Gamma), \\
\label{a11}
J_s = J_s[h,\Gamma] &:=  & \frac{h^{3/2}}{2} \ \partial^3_{x} h  + {h}^{1/2} \ \partial_x \sigma(\Gamma).
\end{eqnarray}
This yields
\begin{equation}\label{a18} 
\begin{array}{l}
\dfrac{\textrm{d}}{\textrm{dt}} \left[ \displaystyle{\int_{0}^1} \left( \dfrac{|\partial_x h|^2}{2} + g_{\sigma}(\Gamma)\right) \text{ d$x$} \right] 
\\[16pt]
\phantom{12345}= 
- \displaystyle{\int_{0}^1} \left[ \dfrac{3}{2} |J_f[h,\Gamma]|^2 + \dfrac{1}{2} |J_s[h,\Gamma]|^2 + \dfrac{1}{24} {h^{3}} | \partial^3_{x} h|^2 + \dfrac{1}{8} {h} | \partial_x \sigma(\Gamma) |^2 \right] \text{ d$x$}  \\[16pt]
\phantom{12345678} + D \displaystyle{\int_{0}^1} \dfrac{\sigma'(\Gamma)}{\Gamma}| \partial_x \Gamma|^2   \text{ d$x$} \,.
\end{array}
\end{equation}
Consequently, we infer that, regardless the qualitative properties of $\sigma,$ $(h,\Gamma)$ should satisfy
\begin{equation} \label{weak.1}
h \in L_\infty(0,T \ ; \  H^1(0,1)), \qquad  \Gamma \in L_\infty(0,T;L_1(0,1))
\end{equation}
together with
\begin{equation} \label{weak.2}
h^{3/2} \partial^3_{x}h \in L_2((0,T) \times (0,1)),    \qquad
h^{1/2} \partial_x \sigma(\Gamma) \in L_2((0,T) \times (0,1)).
\end{equation}
Since $D>0$, the energy estimate \eqref{a18} provides an additional estimate which depends strongly on the properties of $\sigma$, namely $\sqrt{-\sigma'(\Gamma)/\Gamma}\, \partial_x\Gamma\in L_2((0,T)\times (0,1))$. We will actually prove  that, under assumption \eqref{case.2} on $\sigma,$ this additional estimate guarantees that the solutions we construct satisfy the further regularity 
\begin{equation} \label{weak.3}
\Gamma \in L_{2}((0,T) \times (0,1)) \;\;\;\mbox{ and }\;\;\; \sigma(\Gamma)\in L_{4/3}(0,T;W^1_{4/3}(0,1)),
\end{equation}
see {\bf Lemma~\ref{lem.emb}}. Let us point out here that \eqref{case.2} implies that $\sigma'$ does not decay too fast towards $-\infty$ at infinity.  Hence, assuming \eqref{weak.1} and \eqref{weak.2} to hold true, we realize that $J_f[h,\Gamma],J_s[h,\Gamma] \in L_2((0,T) \times (0,1))$. Since an alternative formulation of \eqref{a1}--\eqref{a2} reads,
\begin{eqnarray}
\partial_t h + \partial_x \left( h^{3/2}\, J_f[h,\Gamma] \right) & = & 0 \quad \mbox{ in } \quad (0,\infty)\times (0,1)\,, \label{aaf1}\\
\partial_t \Gamma + \partial_x \left( h^{1/2}\ \Gamma\ J_s[h,\Gamma] \right) & = & D\ \partial_x^2 \Gamma \quad \mbox{ in }\quad (0,\infty)\times (0,1)\,, \label{aaf2}
\end{eqnarray}
we infer from \eqref{weak.1}--\eqref{weak.3} and the embedding of $H^1(0,1)$ in $L_\infty(0,1)$ that $h^{3/2}\ J_f[h,\Gamma]$ and $h^{1/2}\ \Gamma\ J_s[h,\Gamma]$ both belong to $L_1((0,T) \times (0,1))$ and we can give a meaning to \eqref{a1}--\eqref{a2} at least in the following weak sense:

%%%%%%%%%%%%%%%%%%%%%%%%%%%%%%%%%%%%%%%%%%%%%%%%%%%%%%%%%%%%%%%%
\begin{definition}\label{def:ws}
Let $T>0$ and $\sigma$ be a surface tension such that either
\begin{itemize}
\item $\sigma \in \mathcal{C}^1(0,\infty) \cap \mathcal{C}([0,\infty)) $, $\sigma(1)=0$, and \eqref{case.2}
holds true,
\end{itemize}
or
\begin{itemize}
\item $\sigma \in \mathcal{C}^{1,1}(\mathbb R)$ is such that $\sigma(1)=0$ and $g_{\sigma}$, defined in \eqref{a8}, satisfies \emph{(A4)--(A6)}.
\end{itemize} 
Then, given { an initial condition $(h_0,\Gamma_0) \in H^1(0,1) \times L_2(0,1)$ with $h_0\ge 0$ and $\Gamma_0\ge 0$} we say that $(h,\Gamma)$ is a \emph{weak solution} in $(0,T)$  to \eqref{a1}--\eqref{a4} with surface tension $\sigma$
and initial condition $(h_0,\Gamma_0)$, if 
\begin{itemize}
\item { $h\ge 0$ and $\Gamma\ge 0$} satisfy
\begin{equation} \label{aregsolution}
\begin{array}{ll}
h \in L_\infty(0,T;H^1(0,1)) \cap \mathcal{C}([0,T]\times [0,1]), &  \Gamma \in L_\infty(0,T;L_1(0,1)) \cap L_2((0,T) \times (0,1)),\\[8pt]
\partial_x^3 h \in L_2(\mathcal{P}_h(\delta)) \;\;\;\mbox{ for all }\;\;\; \delta>0,    & \sigma(\Gamma) \in L_1(0,T;W_1^1(0,1)) ,\\[8pt]
h^{3/2} \partial^3_{x}h \in L_2(\mathcal{P}_h),    & h^{1/2} \partial_x \sigma(\Gamma) \in L_2((0,T)\times (0,1)),
\end{array}
\end{equation}
where $\mathcal{P}_{h}(\delta) := \{ (t,x)\in (0,T)\times (0,1)\ :\ h(t,x)>\delta \}$ for $\delta>0$ and $\mathcal{P}_{h} := \{ (t,x)\in (0,T)\times (0,1)\ :\ h(t,x)>0 \}.$
\item  for any $\zeta \in {\mathcal{C}}^{\infty}( [0,T] \times [0,1])$ such that 
$\zeta(T,x)= 0$ for all $x \in [0,1]$ and $\partial_x \zeta(t,x) = 0$ for all $(t,x) \in [0,T] \times \{0,1\},$
there holds: 
\begin{equation}
\label{a14}
\displaystyle{\int_{0}^T \int_{0}^{1}} \left( h \ \partial_t \zeta +  \left[\frac{1}{3}\ h^3\, \mathbf{1}_{(0,\infty)}(h)\, \partial_x^3 h + \frac{1}{2}\ h^2\ \partial_x \sigma(\Gamma)  \right]\ \partial_x \zeta \right) \text{d$x$d$s$}  = -\int_{0}^1 h_0(x) \zeta(0,x) \text{d$x$}\,, 
\end{equation}
and
\begin{equation}
\label{a15}
\displaystyle{\int_{0}^T \int_{0}^{1}} \left( \Gamma \ \partial_t \zeta +  \left[ \frac{1}{2}\ h^2\, \mathbf{1}_{(0,\infty)}(h)\,  \Gamma\ \partial_x^3 h + h\ \Gamma\ \partial_x \sigma(\Gamma) \right]\ \partial_x \zeta + D \Gamma \partial_x^2 \zeta \right) \text{d$x$d$s$} = -\int_{0}^1 \Gamma_0(x) \zeta(0,x)\text{d$x$}\,. 
\end{equation} 
\end{itemize} 
\end{definition}
%%%%%%%%%%%%%%%%%%%%%%%%%%%%%%%%%%%%%%%%%%%%%%%%%%%%%%%%%%%%%%%%

With these conventions, our main result reads:

%%%%%%%%%%%%%%%%%%%%%%%%%%%%%%%%%%%%%%%%%%%%%%%%%%%%%%%%%%%%%%%%
\begin{theorem}\label{P.1} %
Let the surface tension $\sigma \in \mathcal{C}^1(0,\infty) \cap \mathcal{C}([0,\infty)) $ satisfy $\sigma(1)=0$ and \eqref{case.2}. Then, given { an initial condition $(h_0,\Gamma_0) \in H^1(0,1) \times L_2(0,1)$ with $h_0\ge 0$, $\Gamma_0\ge 0$ and any $T>0$}, there exists at least one weak solution $(h,\Gamma)$ in $(0,T)$ to \eqref{a1}--\eqref{a4}
with surface tension $\sigma$ and initial condition $(h_0,\Gamma_0)$ in the sense of Definition~\ref{def:ws}.   
\end{theorem} 
%%%%%%%%%%%%%%%%%%%%%%%%%%%%%%%%%%%%%%%%%%%%%%%%%%%%%%%%%%%%%%%%

As mentioned in the Introduction, we split the proof of  {\bf Theorem~\ref{P.1}} into two parts. First, we focus on {the nonnegativity issue} of solutions to \eqref{a1}--\eqref{a2}. In this respect, we go back to the framework considered in \cite{GW06} and we prove:

%%%%%%%%%%%%%%%%%%%%%%%%%%%%%%%%%%%%%%%%%%%%%%%%%%%%%%%%%%%%%%%%
\begin{theorem} \label{P.2}
Let the surface tension $\sigma \in \mathcal{C}^2(\mathbb R)$ be such that $\sigma(1)=0$ and the free energy $g_{\sigma}$ defined by \eqref{a8} satisfies \emph{(A4)--(A6)}. Then, given { an initial condition $(h_0,\Gamma_0) \in H^1((0,1)) \times L_2((0,1))$ with $h_0\ge 0$, $\Gamma_0\ge 0$,  and  any $T>0$,}  there exists at least one weak solution $(h,\Gamma)$ in $(0,T)$ to \eqref{a1}--\eqref{a2} with surface tension $\sigma$ and initial condition $(h_0,\Gamma_0)$ in the sense of Definition~\ref{def:ws}. \\
Moreover, the solution satisfies the further regularity 
\begin{equation} \label{reg.Gamma}
\Gamma \in L_\infty(0,T; L_2(0,1)) \cap L_2(0,T;H^1(0,1)) 
\end{equation}
and the energy estimate
\begin{equation} \label{a2021}
\sup_{t\in [0,T]} \left\{\int_{0}^1 \left[ \dfrac{|\partial_x h(t,x)|^2}{2} + g_\sigma(\Gamma(t,x)) \right] \text{d$x$}\right\}
+ \mathcal{D}[h,\Gamma] \leq 
\int_{0}^1 \left( \dfrac{|\partial_x h_0(x)|^2}{2} + g_{\sigma}(\Gamma_0(x)) \right) \text{d$x$}\,,
\end{equation}
where
\begin{eqnarray*}
\mathcal{D}[h,\Gamma] & := & \int_{0}^T\int_{0}^1 \left( \dfrac{(h^3\, \mathbf{1}_{(0,\infty)}({ h(\tau,x)}))}21 |\partial_x^3 h(\tau,x)|^2 + \frac{h(\tau,x)}8 |\partial_x \sigma({ \Gamma(\tau,x)})|^2  \right)\text{d$x$d$\tau$} \\
& & \qquad\qquad - D\ \int_0^T \int_0^1  \dfrac{\sigma'(\Gamma(\tau,x))}{\Gamma(\tau,x)} |\partial_x \Gamma(\tau,x)|^2\, \text{d$x$d$\tau$}.
\end{eqnarray*}
\end{theorem}
%%%%%%%%%%%%%%%%%%%%%%%%%%%%%%%%%%%%%%%%%%%%%%%%%%%%%%%%%%%%%%%%

With the regularity \eqref{reg.Gamma}, any weak solution constructed in \cite{GW06}  is a weak solution in our sense (see the proof of {\bf Theorem~\ref{P.2}} for further details).  The major novelty in this result is that we obtain nonnegativity of the surfactant concentration $\Gamma.$ For the proof of {\bf Theorem~\ref{P.1}}, we consider a surface tension $\sigma \in \mathcal{C}([0,\infty)) \cap \mathcal{C}^1(0,\infty)$ and introduce a family of approximate surface tensions $(\sigma_k)_{k\in\mathbb N}$ satisfying the assumptions of {\bf Theorem~\ref{P.2}}. We achieve our result by studying the compactness properties of the family of associated weak solutions. A fundamental argument will be that, { owing to assumption \eqref{case.2} and equation \eqref{a18}}, the dissipation of energy is measured by
\begin{equation} \label{eq.diss}
\int_0^T \int_0^1 \dfrac{|\partial_x \Gamma|^2 }{\Gamma (1+\Gamma)^{\theta}}\, \text{d$x$d$\tau$}\,.
\end{equation}
When $\theta \in [0,1),$ this quantity enables us to control $\sqrt{\Gamma}$ in some H\"older space (see {\bf Lemma~\ref{lem.emb}}). This, in turn, yields compactness on the concentration for any bounded family of solutions in $L_2((0,T)\times (0,1))$.

%This strategy does not provide existence of positive weak solutions for the limiting case $\theta = 1.$
%To go further into details, a fundamental argument will be that the dissipation of energy is measured by
%\begin{equation} \label{eq.diss}
%\int_0^T \int_0^1 \dfrac{|\partial_x \Gamma|^2 }{\Gamma (1+\Gamma)^{\theta}}.
%\end{equation}
%When $\theta \in [0,1),$ this $L_2$-norm of $\partial_x\Gamma$ enables to control $\sqrt{\Gamma}$
%in some H\"older space (see {\bf Lemma \ref{lem.emb}}). This yields compactness
%of the concentration of any bounded family of solutions in $L_2((0,T)\times (0,1))$. Unfortunately, in the case $\theta %= 1$ 
%the H\"older space degenerates in a space of continuous functions so that we lose compactness.
%It seems that this threshold is of high importance. Indeed, provided $\sigma'(\Gamma)$ is dominated by $1/(1+\Gamma)$ at %infinity, 
%a good choice of multiplier for \eqref{a1}--\eqref{a2} yields that the integral \eqref{eq.diss} (with $\theta=1$) %measures the dissipation of energy for any small solution. 

%
%
%--------------------------------------------------------------------
%
\section{Existence of nonnegative solutions for decaying surface tensions}\label{sec:ens}
%
%--------------------------------------------------------------------
%
%
In this section, we assume $\sigma \in \mathcal{C}^2(\mathbb R)$ is such that $\sigma(1)=0$ and the free energy $g_{\sigma},$ as defined in \eqref{a8}, satisfies (A4)--(A6) and we construct nonnegative weak solutions to \eqref{a1}--\eqref{a2}. In \cite[Sect.3.4]{GW06}, the authors remark that, for proving nonnegativity of the surfactant concentration of weak solutions to \eqref{a1}--\eqref{a4}, a difficulty arises when multiplying equation \eqref{a2}  by  $\Gamma_- = -\min{\{0,\Gamma\}}.$ Indeed, under assumptions (A4)--(A6), the very low regularity of $\Gamma$ implies only that $\partial_t \Gamma \in L_{3/2}(0,T;(W_3^1(0,1))^*)$ and $\Gamma_-\in L_2(0,T;H^1(0,1))$. This regularity does not allow to define the duality bracket $\langle \partial_t \Gamma, \Gamma_- \rangle$. To construct weak solutions with nonnegative surfactant concentrations, we go back to the strategy applied in \cite{GW06}: construction of solutions to a regularized problem via a Galerkin method, followed by a compactness argument when the regularization parameter goes to $0.$ We introduce a supplementary truncation operator in the regularized problem in order to guarantee that the solutions to the regularized problems have nonnegative surfactant concentrations. 

{Throughout this section,} we fix a nonnegative initial condition $(h_0,\Gamma_0) \in H^1(0,1) \times L_2(0,1).$
We also introduce a  Lipschitz continuous truncation function $\mathcal{T}$ such that
\begin{equation} \label{def.T}
\mathcal{T}(s) =
\left\{
\begin{array}{ll}
s & \text{if $s \in (0,1)$},\\
2-s & \text{if $s \in [1,2]$},\\
0 & \text{if $s\geq 2$},\\
\end{array}
\right.
\qquad \mathcal{T}(-s) = -\mathcal{T}(s) \ \text{if $s<0$},
\end{equation}
and put $\mathcal{T}_k := k\mathcal{T}(\cdot/k)$ for $k \ge 1$. Then, we set
\begin{equation} \label{b1}
\sigma_k(s) := \int_{1}^{s} \mathcal{T}_k(\sigma'(r)) \text{d$r$} \quad \mbox{  for } \, s \in \mathbb R.
\end{equation}
We emphasize that this construction ensures that $\sigma_k \in \mathcal{C}^{1,1}(\mathbb R)$ has bounded first and second derivatives. Associated to this truncation of $\sigma,$ we introduce a truncation of the identity
\begin{equation} \label{b2}
\tau_k(s) := s \ \dfrac{ \sigma_k'(s)}{\sigma'(s)} \quad \mbox{ for } \, s \in \mathbb{R}.
\end{equation}
We note that the construction above is well-defined because 
\begin{equation}\label{spirou}
0 \geq \sigma_k'(s)  \geq \sigma'(s)\quad\mbox{ for all }\quad s \in \mathbb R.
\end{equation} 
With these conventions, our regularized problem reads
\begin{eqnarray}
\label{b3}
\partial_t h + \partial_x \left(\left[a_3(h) + {1}/{k}\right]\ \partial_x^3 h + a_2(h)\ \partial_x \sigma_k(\Gamma) \right) & =& 0\,, \qquad \ \quad (t,x)\in (0,\infty)\times (0,1)\,, \\[4pt]
\label{b4}
 \partial_t \Gamma + \partial_x \left( a_2(h)\ \tau_k(\Gamma)\ \partial_x^3 h + a_1(h)\ \Gamma\ \partial_x \sigma_k(\Gamma) \right) & =& D\ \partial_x^2 \Gamma\,, \quad (t,x)\in (0,\infty)\times (0,1)\,, 
\end{eqnarray}
subject to \eqref{a3}--\eqref{a4}, where $k$ is a positive integer. The { notation $a_i(h)$ stands} for $(\max{\{0,h\}})^i/i$ for $i=1,2,3.$ { This is the same convention} as in \cite{GW06} so that conditions (A1)--(A8) therein are satisfied. 

%--------------------------------------------------------------------
%
\subsection{Existence for \eqref{b3}-\eqref{b4}}\label{sec:ctae}
%
%--------------------------------------------------------------------

To begin with, we fix $k \ge 1$ and prove:

%%%%%%%%%%%%%%%%%%%%%%%%%%%%%%%%%%%%%%%%%%%%%%%%%%%%%%%%%%%%%%%%
\begin{lemma} \label{L1}
Consider { an initial condition $(h_0,\Gamma_0) \in H^1((0,1)) \times L_2((0,1))$ with $h_0\ge 0$ and $\Gamma_0\ge 0.$} For any $k \ge 1$ and $T>0,$ there exists at least a couple of functions $(h,\Gamma)$ having the regularity %
\begin{equation} \label{b3bis}
h \in L_\infty(0,T ; H^1(0,1)) \cap L_2(0,T;H^3(0,1)) \ ,  \quad
\Gamma \in L_\infty(0,T;L_2(0,1)) \cap L_2(0,T;H^1(0,1)),
\end{equation}
\begin{equation} \label{b3ter}
\partial_t h \in L_2(0,T;(H^1(0,1))^*) \ , \quad \partial_t \Gamma \in L_{3/2}(0,T;(W_3^1(0,1))^*),
\end{equation}
and satisfying,
\begin{equation} \label{b5}
\int_0^T \langle \partial_t h, \zeta\rangle\, \text{d$s$} - \int_{0}^T \! \int_{0}^1 \left(a_2(h) \partial_x \sigma_k(\Gamma) + \left[a_3(h) + {1}/{k}\right]\ \partial_x^3 h\right) \partial_x \zeta\, \text{d$x$d$s$} = 0,
\end{equation}
for all $\zeta \in L_2(0,T;H^1(0,1))$, together with
\begin{equation} \label{b6}
\int_0^T \langle \partial_t \Gamma, \zeta\rangle\, \text{d$s$} - \int_{0}^T \! \int_{0}^1  \left(a_1(h)\ \Gamma\ \partial_x \sigma_k(\Gamma)+ a_2(h)\ \tau_k(\Gamma)\ \partial_x^3 h -D\, \partial_x\Gamma \right) \partial_x \zeta\, \text{d$x$d$s$} =0,
\end{equation}
for all $\zeta \in L_3(0,T;W_3^1(0,1))$ and
\begin{equation} \label{b7}
(h(0,\cdot), \Gamma(0,\cdot)) = (h_0,\Gamma_0) \,,
\end{equation}
the latter being meaningful as $h\in \mathcal{C}([0,T];(H^{1}(0,1))^*)$ and $\Gamma\in \mathcal{C}([0,T];(W_3^1(0,1))^*)$ by \eqref{b3bis} and \eqref{b3ter}.\\
Moreover, there holds the energy inequality
\begin{equation} \label{b5bis}
\sup_{t\in [0,T]} \left\{ \int_{0}^1 \left[ \dfrac{|\partial_x h(t,x)\vert^2}{2} + g_{\sigma}(\Gamma(t,x)) \right]\,\text{d$x$}\right\} + \tilde{\mathcal{D}_k[h,\Gamma]} \leq
\int_{0}^1 \left[ \dfrac{|\partial_x h_0(x)|^2}{2} + g_{\sigma}({ \Gamma_0(x)}) \right] \text{d$x$} ,
\end{equation}
where
$$
\tilde{\mathcal{D}}_k[h,\Gamma] := \int_0^T\int_{0}^1 \left\{ \left[ \frac{1}{k}  + \frac{a_3(h)}{7} \right] |\partial_x^3 h|^2 - D \dfrac{\sigma'(\Gamma)}{\Gamma} |\partial_x \Gamma|^2 + \frac{a_1(h)}{8}  |\partial_x \sigma_k(\Gamma)|^2 \right\}\, \text{d$x$d$s$}. 
$$
\end{lemma}
%%%%%%%%%%%%%%%%%%%%%%%%%%%%%%%%%%%%%%%%%%%%%%%%%%%%%%%%%%%%%%%%

%%%%%%%%%%%%%%%%%%%%%%%%%%%%%%%%%%%%%%%%%%%%%%%%%%%%%%%%%%%%%%%%
\begin{remark}\label{re:approx}%
Note that, in \eqref{b5bis}, $\sigma_k$ only appears in the last term of $\tilde{\mathcal{D}}_k[h,\Gamma]$.
\end{remark}
%%%%%%%%%%%%%%%%%%%%%%%%%%%%%%%%%%%%%%%%%%%%%%%%%%%%%%%%%%%%%%%%

\begin{proof}
We follow here the Galerkin method from \cite[Section~3]{GW06}. The system \eqref{b3}--\eqref{b4} is actually almost identical to the regularized system used in \cite[Section~3]{GW06} except that the truncation function $\tau_k$ is replaced by the identity there. Since $\tau_k$ is a bounded and Lipschitz continuous function, the analysis performed in \cite[Section~3]{GW06} carries over to \eqref{b3}--\eqref{b4} with only slight changes, the main one arising in the derivation of the energy inequality. We will thus only give a sketch of the proof and refer to \cite[Section~3]{GW06} for details.\\
The first step is an alternative formulation of \eqref{b3}--\eqref{b4} in therms of $h$ and the new unknown function $v:=g_\sigma'(\Gamma)$, the latter being well-defined thanks to the convexity (A5) of $g_\sigma$. Denoting the inverse function of $g_\sigma'$ by $W$, we have
\begin{eqnarray}
\label{b3af}
\partial_t h + \partial_x \left(\left[a_3(h) + {1}/{k}\right]\ \partial_x^3 h - a_2(h)\tau_k(W(v))\partial_x v \right) & =& 0\,,  \\[4pt]
\label{b4af}
 \partial_t W(v) + \partial_x \left( a_2(h)\ \tau_k(W(v))\ \partial_x^3 h - a_1(h)\ W(v) \tau_k(W(v)) \partial_x v  \right) & =& D\ \partial_x^2 W(v)\,, 
\end{eqnarray}
in $(0,\infty)\times (0,1).$ As already mentioned, \eqref{b3af}--\eqref{b4af} is the same as the system studied in \cite[Section~3]{GW06} except for the terms involving { the bounded and Lipschitz continuous function $\tau_k$}. Not surprisingly, considering the same Galerkin approximation to \eqref{b3af}--\eqref{b4af} as in \cite[Section~3]{GW06}, one can prove the local existence of solutions to the Galerkin approximations exactly in the same way as in \cite[Section~3.1]{GW06}. To obtain the global existence, we argue as in \cite[Section~3.2]{GW06} by deriving an energy estimate for the Galerkin approximations. Since there is a slight modification { necessary}, let us sketch the proof for \eqref{b3af}--\eqref{b4af}, the argument being the same at the level of the Galerkin approximations. We multiply \eqref{b3af} by $-\partial_x^2 h$, \eqref{b4af} by $v=g_\sigma'(\Gamma)$, integrate over $(0,1)$, and add the resulting identities to obtain
$$
\dfrac{\textrm{d}}{\textrm{d$t$}} \int_{0}^1 \left[ \dfrac{|\partial_x h|^2}{2} + g_{\sigma}(\Gamma) \right]\, \text{d$x$} + \int_{0}^1 \left\{ \left[ \frac{1}{k}  + a_3(h)\right] |\partial_x^3 h|^2 - D \, \dfrac{\sigma'(\Gamma)}{\Gamma} |\partial_x \Gamma|^2 + a_1(h) { (\sigma'(\Gamma) \sigma'_k(\Gamma))}|\partial_x \Gamma|^2 \right\}\, \text{d$x$} = I,
$$
where (see \eqref{a8} and \eqref{b2})
$$
I := - \int_{0}^1 a_2(h) \  \partial_x \Gamma \ \partial_x^3h \left( \sigma'_k(\Gamma) + \sigma'(\Gamma) \dfrac{\tau_k(\Gamma)}{\Gamma}\right)\, \text{d$x$} = -\int_{0}^1 2a_2(h)   \ \partial_x \sigma_k(\Gamma) \ \partial_x^3h \, \text{d$x$} . 
$$
Since $\sigma_k'\le 0$, it follows from \eqref{spirou} that $(\sigma_k')^2\le \sigma' \sigma_k'$ while Young's inequality ensures that 
$$
|2a_2(h) \ \partial_x \sigma_k(\Gamma) \ \partial_x^3h |  \leq \frac{7 a_1(h)}{8}\ |\partial_x \sigma_k(\Gamma)|^2 + \frac{6 a_3(h)}{7}\ |\partial_x^3h|^2\,,
$$

so that we finally obtain
\begin{multline*}
\dfrac{\textrm{d}}{\textrm{d$t$}} \int_{0}^1 \left[ \dfrac{|\partial_x h|^2}{2} + g_{\sigma}(\Gamma) \right]\, \text{d$x$}
+ \int_{0}^1 \left\{ \left[ \frac{1}{k}  + a_3(h)\right] |\partial_x^3 h|^2 - D \dfrac{\sigma'(\Gamma)}{\Gamma} |\partial_x \Gamma|^2 + a_1(h) |\partial_x \sigma_k(\Gamma)|^2 \right\}\, \text{d$x$}\\
\leq \int_0^1 \left[ \frac{7 a_1(h)}{8}\ |\partial_x \sigma_k(\Gamma)|^2 + \frac{6 a_3(h)}{7}\ |\partial_x^3h|^2 \right]\, \text{d$x$} \,.
\end{multline*}
This yields
$$
\dfrac{\textrm{d}}{\textrm{d$t$}} \int_{0}^1 \left[ \dfrac{|\partial_x h|^2}{2} + g_{\sigma}(\Gamma) \right]\, \text{d$x$}
+ \int_{0}^1 \left\{ \left[ \frac{1}{k}  + \frac{a_3(h)}{7} \right] |\partial_x^3 h|^2 - D \dfrac{\sigma'(\Gamma)}{\Gamma} |\partial_x \Gamma|^2 + \frac{a_1(h)}{8}\  |\partial_x \sigma_k(\Gamma)|^2 \right\}\, \text{d$x$} \leq 0\,,
$$
whence \eqref{b5bis} after time integration.\\
The convergence of the Galerkin approximations to a solution to \eqref{b3}--\eqref{b4} satisfying the properties listed in {\bf Lemma~\ref{L1}} is then carried out as in \cite[Section~3.3]{GW06} to which we refer.
\end{proof}

At this point, we show that the idea to introduce truncation functions  $\tau_k$ and $\sigma_k$
yields the nonnegativity of $\Gamma$. This relies on a gain of regularity for $\partial_t \Gamma.$ 

%%%%%%%%%%%%%%%%%%%%%%%%%%%%%%%%%%%%%%%%%%%%%%%%%%%%%%%%%%%%%%%%
\begin{lemma} \label{L2}
Consider { an initial condition $(h_0,\Gamma_0) \in H^1((0,1)) \times L_2((0,1))$ with $h_0\ge 0$ and $\Gamma_0\ge 0.$} Given $k \ge 1$ and $T>0$, any solution $(h,\Gamma)$ to \eqref{b3bis}-\eqref{b5bis} and \eqref{a3}--\eqref{a4} in the sense of {\bf Lemma~\ref{L1}} satisfies $\partial_t \Gamma \in L_2(0,T;(H^1(0,1))^*)$ and $\Gamma \geq 0$ a.e. in $(0,T) \times (0,1).$
\end{lemma}
%%%%%%%%%%%%%%%%%%%%%%%%%%%%%%%%%%%%%%%%%%%%%%%%%%%%%%%%%%%%%%%%

\begin{proof}
Owing to \eqref{b3bis}, the embedding of $H^1(0,1)$ in $L_\infty(0,1)$, and the compactness of the supports of $\sigma_k'$ and $\tau_k$ (which follows from (A5) and the properties of $\mathcal{T}$), there holds
$$
a_1(h)\ \Gamma\ \partial_x \sigma_k(\Gamma)+ a_2(h)\ \tau_k(\Gamma)\ \partial_x^3 h = 
 a_1(h) \Gamma \sigma'_k(\Gamma) \partial_x \Gamma + a_2(h) \ \tau_k(\Gamma) \ \partial_x^3 h 
 \in L_2((0,T) \times  (0,1)),
$$
and $D \partial_x \Gamma\in L_2((0,T) \times  (0,1))$.  As a consequence \eqref{b6} also holds true for all $\zeta\in L_2(0,T;H^1(0,1))$ and $\partial_t \Gamma \in L_2(0,T;(H^1(0,1))^*)$. Then, if $\beta \in \mathcal{C}^2(\mathbb{R})$ is such that $\beta'$ is Lipschitz continuous, we have $\beta'(\Gamma)\in L_2(0,T;H^1(0,1))$ and 
$$
\dfrac{d}{dt} \int_0^1 \beta(\Gamma)\, \text{d$x$} = \langle \partial_t \Gamma , \beta'(\Gamma) \rangle\,.
$$
Assuming furthermore that $\beta$ is convex, { i.e. $\beta'' \geq 0$}, there holds, for any $t\in (0,T)$,
$$
\int_{0}^1 \beta(\Gamma(t))\, \text{d$x$} \leq \int_{0}^1 \beta(\Gamma_0)\, \text{d$x$} + \int_0^T \int_{0}^1 \left| \left(a_1(h)\ \Gamma\ \partial_x \sigma_k(\Gamma)+ a_2(h)\ \tau_k(\Gamma)\ \partial_x^3 h\right) \beta''(\Gamma) \partial_x \Gamma  \right|\, \text{d$x$d$s$}. 
$$
To finish off the proof, we apply this inequality to a family of functions approximating the negative part of $\Gamma.$ Namely, we fix a nonnegative $\chi \in \mathcal{C}^{\infty}_0(\mathbb R)$ such that $\chi \not\equiv 0$ has support in $(-1,0)$ and define 
$\beta_1$ by
$$
\beta_1(0) = 0, \quad \beta_1'(s) := - \dfrac{\int_s^{\infty} \chi(\alpha) \text{d$\alpha$}}{\int_{-{\infty}}^{\infty} \chi(\alpha) \text{d$\alpha$}} \quad\mbox{ for }\quad s\in\mathbb{R}\,.
$$  
We then set $\beta_{\varepsilon}(s) := \varepsilon\beta_1(s/\varepsilon)$ for $s\in\mathbb{R}$ and $\varepsilon >0.$ Taking $\beta=\beta_\varepsilon$ for $\varepsilon >0$ in the above inequality,  there holds, for each $t\in (0,T)$,
$$
\int_{0}^1 \beta_{\varepsilon}(\Gamma(t))\, \text{d$x$}  \leq \int_0^T \int_{0}^1 \left| \left(a_1(h)\ \Gamma\ \partial_x \sigma_k(\Gamma)+ a_2(h)\ \tau_k(\Gamma)\ \partial_x^3 h\right) \beta_{\varepsilon}''(\Gamma) \partial_x \Gamma  \right|\, \text{d$x$d$s$} \,,
$$
{ since} $\beta_{\varepsilon}(\Gamma_0) = 0$ due to $\Gamma_0\ge 0$. Observing that $|\tau_k(s) \beta_\varepsilon''(s)| \le | s \beta_\varepsilon''(s)| \le C(\chi)$ for $s\in\mathbb{R}$, we have 
\begin{multline*}
 \int_0^T \!\! \int_{0}^1 \left| \left(a_1(h)\ \Gamma\ \partial_x \sigma_k(\Gamma)+ a_2(h)\ \tau_k(\Gamma)\ \partial_x^3 h\right) \beta_{\varepsilon}''(\Gamma) \partial_x \Gamma  \right| \, \text{d$x$d$s$} 
\\
\leq C (\chi) \int_{\{|\Gamma| < \varepsilon\}}  \left[ \left| a_1(h)\  \partial_x \sigma_k(\Gamma) \partial_x \Gamma\right| + \left|  a_2(h)\  \partial_x^3 h \  \partial_x \Gamma\right|\right]\, \text{d$x$d$s$}  \phantom{12345678901}\\
\leq C (\chi) \left[  \|a_1(h) \sigma'_k(\Gamma)\|_{L_\infty(0,T; L_\infty}(0,1) ) + \|a_2(h) \|_{L_\infty(0,T; L_\infty}(0,1) )\right] \int_{\{|\Gamma| < \varepsilon\}}  \left[ \left| \partial_x \Gamma\right|^2 + \left|  \partial_x^3 h \  \partial_x \Gamma\right|\right]\, \text{d$x$d$s$} .
\end{multline*}
As $\partial_x \Gamma$ and $\partial_x^3h$ both belong to $L_2((0,T)\times (0,1))$ and $\partial_x\Gamma=0$ a.e. in $\{\Gamma=0\}$ by \cite[Lemma~A.4]{KS00}, we obtain in the limit $\varepsilon \to 0$
$$
\int_{0}^1 \max{\{-\Gamma(t,x),0\}}\, \text{d$x$} \leq 0 ,\quad \mbox{ for }\quad t \in (0,T).
$$ 
This completes the proof.
\end{proof}
 
%--------------------------------------------------------------------
%
\subsection{Proof of {\bf Theorem~\ref{P.2}}} \label{sec:cvgce}
%
%--------------------------------------------------------------------

Let $\sigma$ be as in the statement of {\bf Theorem~\ref{P.2}} and consider { an initial condition $(h_0,\Gamma_0) \in H^1((0,1)) \times L_2((0,1))$ with $h_0\ge 0$, $\Gamma_0\ge 0$} and $T>0.$ First, applying  {\bf Lemma~\ref{L1}} and {\bf Lemma~\ref{L2}}, we obtain a sequence $(h_k,\Gamma_k)_{k\ge 1}$ of solutions to \eqref{b3}--\eqref{b4}, \eqref{a3}--\eqref{a4} for which $\partial_t\Gamma_k \in L_2(0,T;(H^1(0,1))^*)$ and $\Gamma_k \geq 0$ a.e. in $(0,T)\times (0,1).$ In particular, for each $k\ge 1,$ the time regularity of $h_k$ and $\Gamma_k,$ together with the initial conditions \eqref{b7} $(h_k(0,\cdot),\Gamma_k(0,\cdot)) = (h_0,\Gamma_0),$  yield the integration by parts formula:
$$
\int_0^T \langle \partial_t h_k , \zeta \rangle\, \text{d$t$} = - \int_0^1 h_0(x) \zeta(0,x) \text{d$x$} - \int_0^T \int_0^1 h_k(s,x)\partial_t \zeta(s,x) \text{d$x$d$t$}\,,
$$
$$
\int_0^T \langle \partial_t \Gamma_k , \zeta \rangle\, \text{d$t$} = - \int_0^1 \Gamma_0(x) \zeta(0,x) \text{d$x$} - \int_0^T \int_0^1 \Gamma_k(s,x)\partial_t \zeta(s,x) \text{d$x$d$t$}\,
$$
for any test function $\zeta \in \mathcal{C}^{\infty}([0,T] \times [0,1])$ such that $\zeta(T,x)= 0$ for all $x \in [0,1]$ and $\partial_x \zeta(t,x) = 0$ for all $(t,x) \in [0,T] \times \{0,1\}$. Hence, taking such a test function $\zeta$ in \eqref{b5}-\eqref{b6} we obtain :
\begin{equation} \label{b5prime}
\displaystyle{\int_{0}^T \int_{0}^{1}} \left( h_k \ \partial_t \zeta +  \left[\left( a_3(h_k) + \frac{1}{k}\right) \ \partial_x^3 h_k + a_2(h_k)\ \partial_x \sigma_k(\Gamma_k)  \right]\ \partial_x \zeta \right)\, \text{d$x$d$t$}   = -\int_{0}^1 h_0(x) \zeta(0,x) \text{d$x$}\,, 
\end{equation}
\begin{equation} \label{b6prime}
\displaystyle{\int_{0}^T \int_{0}^{1}} \left( \Gamma_k \ \partial_t \zeta +  \left[ a_2(h_k)\ \tau_k(\Gamma_k)\ \partial_x^3 h_k + a_1(h_k)\ \Gamma_k\ \partial_x \sigma_k(\Gamma_k) \right] \partial_x \zeta + D \Gamma_k \partial_x^2 \zeta \right)\, \text{d$x$d$t$}  = -\int_{0}^1 \Gamma_0(x) \zeta(0,x)\text{d$x$}\,. 
\end{equation}
So, the proof reduces to find a weak cluster point $(h,\Gamma)$ of the sequence $((h_k,\Gamma_k))_{k\ge 1}$ that has the regularity \eqref{aregsolution} and for which we can pass to the limit in the two previous equations.

First, we note that the conservation laws \eqref{aconserver} are also satisfied by $(h_k,\Gamma_k).$ Consequently, due to \eqref{b5bis} and the Poincar\'e inequality, we have uniform bounds for
\begin{itemize}
\item $(h_k)_{k\ge 1}$ in $L_\infty(0,T;H^1(0,1))$ and $(g_\sigma(\Gamma_k))_{k\ge 1}$ in $L_\infty(0,T;  L_1(0,1)),$ 
\item $(\sqrt{a_3(h_k)} \partial_x^3h_k )_{k\ge 1}$,   $(\sqrt{a_1(h_k)} \partial_x \sigma_k(\Gamma_k))_{k\ge 1},$ and
$(\sqrt{-\sigma'(\Gamma_k)/\Gamma_k} \partial_x \Gamma_k)_{k\ge 1}$ in $L_2((0,T) \times (0,1)).$
\end{itemize}
Owing to the bound (A5) from below on $\sigma',$ this yields a uniform bound on $(\Gamma_k)_{k\ge 1}$ in  $L_\infty(0,T;L_2(0,1))$ and $L_2(0,T;H^1(0,1))$, and the sequence of fluxes, given by
\begin{eqnarray*}
J_s^k & := & \frac{a_2(h_k)}{a_1(h_k)^{1/2}}\ \frac{\tau_k(\Gamma_k)}{\Gamma_k}\ \partial_x^3 h_k + a_1(h_k)^{1/2}\ \partial_x \sigma_k(\Gamma_k), \\ 
J_f^k & := & \left( \frac{a_3(h_k)}{3} + \frac{1}{3k} \right)^{1/2}\ \partial_x^3 h_k + a_2(h_k)\ \left( 3 a_3(h_k) + \frac{3}{k} \right)^{-1/2}\ \partial_x \sigma_k(\Gamma_k), 
\end{eqnarray*}
are also bounded in $L_2((0,T) \times (0,1))$ by \eqref{b5bis}. 
 
Repeating the arguments in \cite[Section~2]{BF90} and \cite[Section~3.4]{GW06}, we may extract a subsequence (not relabeled) and find functions $h$ and $\Gamma$ such that the following convergences hold:
\begin{itemize}
\item$h_k \to h \text{ in $\mathcal{C}([0,T]\times [0,1])$}$ and
         $\Gamma_{k} \to \Gamma \text{ in $L_2(0,T;L_p(0,1))$ for all $p \in [1,\infty)$}$,
\item $\sqrt{3a_3(h_{k}) + 3/k }\, \partial_x^3 h_k \rightharpoonup H$ in $L_2((0,T) \times (0,1))$ with $H=\sqrt{3a_3(h)}\, \partial_x^3 h$ a.e. in $\{h\ne 0\}$,
\item $\partial_x \Gamma_k \rightharpoonup \partial_x \Gamma \text{ in $L_2((0,T) \times (0,1))$}.$
\end{itemize}
 Arguing as in the proof of \cite[Equation~(3.28)]{GW06}, the previous convergences imply that
  { 
 $$(3a_3(h_{k}) + 3/k )\, \partial_x^3 h_k \rightharpoonup h^{3/2} \mathbf{1}_{(0,\infty)}(h) \, \partial_x^3 h\quad\text{in}\quad L_2((0,T) \times (0,1))\ .$$
} Next,
interpolating the bounds on $(\Gamma_k)_{k\ge 1}$ with the help of \cite[Proposition~I.3.2]{DiB93}, we deduce that $(\Gamma_k)_{k\ge 1}$ is bounded in $L_6((0,T)\times (0,1))$ and that the convergence of $(\Gamma_k)_{k\ge 1}$ to $\Gamma$ takes actually place in $L_p((0,T) \times (0,1))$ for all $p\in[2,6)$. Now, since
$$
0 \leq \tau_k(s) \leq s \quad \mbox{ for all }\quad s \geq 0 \quad \mbox{ and }\quad \tau_k(s)=s\quad \mbox{ for }\quad 0\leq s \leq s_k:= \left[ (k/C_{g})^{r/(r+1)}- 1 \right]^{1/r}
$$ 
(see assumption (A6) and \eqref{b2} for the definitions of $r$ and $\tau_k$, respectively), we have, for $p\ge 1$,
\begin{eqnarray*}
\int_0^T \int_0^1 \left| \frac{\tau_k(\Gamma_k)}{\Gamma_k} - 1 \right|^p\, \text{d$x$d$t$}  & = & \int_{\{\Gamma_k> s_k\}} \left|\frac{\tau_k(\Gamma_k)}{\Gamma_k} - 1 \right|^p\, \text{d$x$d$t$} \le 2^p\ \int_{\{\Gamma_k>s_k\}} \text{d$x$d$t$} \\
& \le & \frac{2^p}{s_k^6}\ \int_{\{\Gamma_k>s_k\}}  \Gamma_k^6\, \text{d$x$d$t$} \le \frac{C(p,T)}{s_k^6}\,. 
\end{eqnarray*}
Since $s_k\to \infty$ as $k\to\infty$, we conclude that $\tau_k(\Gamma_k)/\Gamma_k\to 1$ in $L_p((0,T) \times (0,1))$ for any $p\ge 1.$ Similarly, since $\sigma_k'(s)=\sigma'(s)$ for $s\in [0,s_k]$  and $\sigma'\in \mathcal{C}^1(\mathbb{R}),$ it follows from (A6) and \eqref{spirou} that, given $p_0\in [1,6/(r+1))$, $R\ge 1$, and $k\ge 1$ such that $s_k\ge R$, we have 
\begin{eqnarray*}
& & \int_0^T \int_0^1 \left| \sigma'_k(\Gamma_k) - \sigma'(\Gamma) \right|^{p_0}\, \text{d$x$d$t$} \\
& \le & \int_{\{\max{\{\Gamma_k,\Gamma\}}\le R\}} \left| \sigma'(\Gamma_k) - \sigma'(\Gamma) \right|^{p_0}\, \text{d$x$d$t$} + \int_{\{\Gamma_k>R\}\cup \{\Gamma>R\}} \left| \sigma'_k(\Gamma_k) - \sigma'(\Gamma) \right|^{p_0}\, \text{d$x$d$t$} \\
& \le & \|\sigma''\|_{L_\infty(0,R)}\ \int_{\{\max{\{\Gamma_k,\Gamma\}}\le s_k\}} \left| \Gamma_k -\Gamma \right|^{p_0}\, \text{d$x$d$t$} \\
& + & C(p_0,C_g,r)\ \int_{\{\Gamma_k>R\}\cup \{\Gamma>R\}} \left( \Gamma_k^{(r+1)p_0} + \Gamma^{(r+1)p_0} \right)\, \text{d$x$d$t$} \\
& \le & \|\sigma''\|_{L_\infty(0,R)}\ \int_0^T \int_0^1 \left| \Gamma_k -\Gamma \right|^{p_0}\, \text{d$x$d$t$} + \frac{C(p_0,C_g,r)}{R^{6-(r+1)p_0}}\ \int_{\{\Gamma_k>R\}\cup \{\Gamma>R\}} \left( \Gamma_k^6 + \Gamma^6 \right)\, \text{d$x$d$t$}  \\
& \le & \|\sigma''\|_{L_\infty(0,R)}\ \int_0^T \int_0^1 \left| \Gamma_k -\Gamma \right|^{p_0}\, \text{d$x$d$t$} + \frac{C(p_0,C_g,r,T)}{R^{6-(r+1)p_0}}\,.
\end{eqnarray*}
Letting first $k\to\infty$ and then $R\to\infty$ yield that $\sigma'_k(\Gamma_k) \to \sigma'(\Gamma)$ in $L_{p_0}((0,T)\times (0,1))$ for any $p_0 \in [1,6/(r+1)).$ As $r<2$ we note that we may choose $p_0>2$ in the previous convergence which, combined with the weak convergence of $ (\partial_x\Gamma_k)_{k\ge 1}$ in $L_2((0,T)\times (0,1))$ implies that $\partial_x \sigma_k(\Gamma_k) \rightharpoonup \partial_x\sigma(\Gamma)$ in $L_{q_0}((0,T) \times (0,1))$ for some $q_0>1.$ Consequently, $\sqrt{a_1(h_k)}\ \partial_x \sigma_k(\Gamma_k) \rightharpoonup \sqrt{a_1(h)}\partial_x \sigma(\Gamma)$ and 
$$
a_2(h_k)\ \left( 3 a_3(h_k) + \frac{3}{k} \right)^{-1/2}\ \partial_x \sigma_k(\Gamma_k) \rightharpoonup \frac{\sqrt{a_1(h)}}{2}\ \partial_x \sigma(\Gamma)\quad \mbox{ in }\quad L_2((0,T) \times (0,1)).
$$ 
Thus, we conclude that $J_s^k \rightharpoonup H/2 +\sqrt{a_1(h)}\partial_x \sigma(\Gamma)$ and $J_f^k \rightharpoonup H + (\sqrt{a_1(h)}/2)\partial_x \sigma(\Gamma)$ in $L_2((0,T)\times (0,1)).$ Combining these convergences with the convergence of $(h_k)_{k\ge 1}$ to $h$ in $\mathcal{C}([0,T]\times [0,1])$ and that of $(\Gamma_k)_{k\ge 1}$ to $\Gamma$ in $L_2((0,T) \times (0,1))$ allows us to pass to the limit in \eqref{b5prime}--\eqref{b6prime}.

That $h$ is nonnegative can be obtained as in \cite[Section~3.4]{GW06} while the nonnegativity of $\Gamma$ is preserved by the weak limit.  Concerning the energy estimate \eqref{a2021}, we { recall (A5) and} prove as above that
$$
\sqrt{-\sigma'_k(\Gamma_k)/\Gamma_k} \to \sqrt{-\sigma'(\Gamma)/\Gamma}   \text{ in $L_2((0,T) \times (0,1)).$}
$$
Consequently, $(\sqrt{-\sigma'_k(\Gamma_k)/\Gamma_k}\partial_x \Gamma_k)_{k\ge 1}$ converges weakly in $L_1((0,T)\times (0,1))$ to $\sqrt{-\sigma'(\Gamma)/\Gamma}\partial_x \Gamma$, and we can pass to the weak limit in the energy estimate. This completes the proof of {\bf Theorem~\ref{P.2}}. 
%
%
%--------------------------------------------------------------------
%
\section{Existence of nonnegative solutions for slowly decaying surface tension} \label{sec:stabsigma}
%
%--------------------------------------------------------------------
%
%
From {\bf Theorem~\ref{P.2}} we obtain existence of weak solutions for a class of surface tension $\sigma$ decreasing at least quadratically to $-\infty$.  We  { now} extend with {\bf Theorem~\ref{P.1}} the existence result to a class containing surface tensions which decrease slowly to $-\infty$ at infinity (but not too { slowly}, see (H2)) and are thus closer to applications. To this end, we fix a surface tension $\sigma$ satisfying (H1)--(H2) and { an initial condition  $(h_0,\Gamma_0) \in H^1(0,1) \times L_2(0,1)$ satisfying $h_0\ge 0$ and $\Gamma_0\ge 0.$} We split the proof of Theorem~\ref{P.1} into three steps: we first construct a sequence $(\sigma_k)_{k\ge 1}$ of surface tensions approximating $\sigma$ and enjoying the properties (A4)-(A6) for { $k\ge 4$} (with constants depending of course on $k$). Owing to this construction, we may apply Theorem~\ref{P.2} to obtain, for each $k\ge { 4}$, a nonnegative weak solution $(h_k,\Gamma_k)$ to \eqref{a1}--\eqref{a4} satisfying \eqref{a2021}. We then show that $(h_k,\Gamma_k)_{k\ge { 4}}$ is compact in suitable function spaces. In the last step, we identify the equations satisfied by the cluster points $(h,\Gamma)$ of $(h_k,\Gamma_k)_{k\ge { 4}}.$ 

%%%%%%%%%%%%%%%%%%%%%%%%%%%%%%%%%%%%%%%%%%%%%%%%%%%%%%%%%%%%%%%%
\subsection{Construction of approximate surface tensions}\label{sec:cast}
%%%%%%%%%%%%%%%%%%%%%%%%%%%%%%%%%%%%%%%%%%%%%%%%%%%%%%%%%%%%%%%%

For $k\ge 1,$ we set $\tilde{\sigma}_k(1):=0$ and
\begin{equation} \label{def.skt}
\tilde{\sigma}'_k(s) 
:= 
\left\{
\begin{array}{ccl}
\displaystyle{\left[k \sigma'\left( \frac{1}{k}\right) - k \right] s} & \text{ for } & \displaystyle{s < \frac{1}{k},} \\[8pt]
 \sigma'(s)  & \text{ for } & \displaystyle{\frac{1}{k} \leq s \leq k,} \\[8pt]
\displaystyle{\sigma'(k) - \frac{s}{{k}^{1+{\theta}}}} & \text{ for } & s > k.
\end{array}
\right.
\end{equation}
Recall that $\theta$ is defined in \eqref{case.2}. Denoting  a family of even mollifiers by $(\chi_{\varepsilon})_{\varepsilon >0}$, we introduce then the approximate { surface tension $\sigma_k$ by}
\begin{equation} \label{def.sk}
\sigma'_k : =  \chi_{1/k^2} * \tilde{\sigma}'_k, \qquad \sigma_k(1) = 0.
\end{equation}
The following proposition verifies that we can apply {\bf Theorem~\ref{P.2}} to any approximate surface tension.

%%%%%%%%%%%%%%%%%%%%%%%%%%%%%%%%%%%%%%%%%%%%%%%%%%%%%%%%%%%%%%%%
\begin{proposition} \label{pro.tec}
Given $k \ge 4,$ the free energy $g_k:=g_{\sigma_k}$ associated to $\sigma_k$ via formula \eqref{a8} satisfies
(A4)--(A6).
\end{proposition}
%%%%%%%%%%%%%%%%%%%%%%%%%%%%%%%%%%%%%%%%%%%%%%%%%%%%%%%%%%%%%%%%

\begin{proof}
By construction, $\sigma'_k \in \mathcal{C}^{\infty}(\mathbb R)$ and, owing to the properties of $\chi_{1/k^2}$, straightforward computations yield 
that 
\begin{eqnarray}
\sigma'_k(s) & = & \left[ k \sigma'\left( \dfrac{1}{k} \right) -k \right]\ s  \quad \mbox{ for all } \, s \leq \dfrac{k-1}{k^2}, \label{fantasio} \\
\sigma'_k(s) & = & \sigma'(k) - \dfrac{s}{k^{1+\theta}} \quad \mbox{ for all } \, s \geq k+\dfrac{1}{k^2}. \label{spip}
\end{eqnarray}
In particular, it follows from \eqref{fantasio} that $[s \mapsto \sigma_k'(s)/s ]\in \mathcal{C}^{\infty}(\mathbb R)$ and $g_k$ satisfies (A4). Next, if $s\in ((k-1)/k^2,k+(1/k^2))$, we have $s-(1/k^2)\ge (k-2)/k^2\ge 1/(2k)$ and it follows from { \eqref{def.skt},} (H1), and (H2) that
\begin{eqnarray}
\sigma_k'(s) & \le & \int_{\mathbb{R}} \left[ k \sigma'\left( \frac{1}{k} \right)\ r\ \mathbf{1}_{(0,1/k)}(r) + \sigma'(r)\ \mathbf{1}_{(1/k,k)}(r) + \sigma'(k)\ \mathbf{1}_{(k,\infty)}(r) \right]\ \chi_{1/k^2}(s-r)\, \text{d$r$} \nonumber\\
& \le & \int_{\mathbb{R}} \left[ \frac{1}{2}\ \sigma'\left( \frac{1}{k} \right)\ \mathbf{1}_{(0,1/k)}(r) + \frac{1}{2}\ \sigma'(r)\ \mathbf{1}_{(1/k,k)}(r) + \frac{1}{2}\ \sigma'(k)\ \mathbf{1}_{(k,\infty)}(r) \right]\ \chi_{1/k^2}(s-r)\, \text{d$r$} \nonumber\\
& \le & \frac{1}{2}\ \sup_{[1/k,k]}{\{\sigma'\}}\ \int_{\mathbb{R}} \chi_{1/k^2}(s-r)\, \text{d$r$} = \frac{1}{2}\ \sup_{[1/k,k]}{\{\sigma'\}} < 0\,. \label{gaston}
\end{eqnarray}
We then infer from \eqref{fantasio}--\eqref{gaston} that
$$
g_k''(s) = -\frac{\sigma_k'(s)}{s} \ge \left\{
\begin{array}{ccl}
k & \text{ if } & s \le \displaystyle{\frac{k-1}{k^2}}\,, \\
\displaystyle{-\frac{1}{4k}\ \sup_{[1/k,k]}{\{\sigma'\}}} & \text{ if } & \displaystyle{\frac{k-1}{k^2} \le s \le k+\frac{1}{k^2}} \,, \\
\displaystyle{\frac{1}{k^{1+\theta}}} & \text{ if } & \displaystyle{k+\frac{1}{k^2} < s}\,,
\end{array}
\right.
$$
and we obtain the existence of a constant $c_k >0$ for which (A5) holds. Finally, it follows from \eqref{case.2} that, for $s\in ((k-1)/k^2,k+(1/k^2))$, 
\begin{eqnarray*}
\sigma_k'(s) & \ge & - \int_{\mathbb{R}} \left[ (1+\sigma_0)\ \mathbf{1}_{(0,1/k)}(r) + \sigma_0\ \mathbf{1}_{(1/k,k)}(r) + \left( \sigma_0 + \frac{r}{k^{1+\theta}} \right)\ \mathbf{1}_{(k,\infty)}(r) \right]\ \chi_{1/k^2}(s-r)\, \text{d$r$} \\
& \ge & - (2+\sigma_0)\ \int_{\mathbb{R}} \chi_{1/k^2}(s-r)\, \text{d$r$} = -(2+\sigma_0)\,.
\end{eqnarray*}
Noting that \eqref{case.2} and \eqref{fantasio}--\eqref{spip} guarantee { this lower bound also} for $s\in [0,(k-1)/k^2)$ and $s\ge k+(1/k^2)$, we conclude that
\begin{equation}
\sigma_k'(s) \ge -(2+\sigma_0) \quad \mbox{ for }\quad s\ge 0\,. \label{volvic}
\end{equation}
In addition, it follows from \eqref{case.2} and \eqref{fantasio} that $\sigma_k'(s)/s \ge -(1+\sigma_0)\ k$ for $s\in (-\infty,(k-1)/k^2]$. These two facts give 
\begin{equation} \label{bound.sigma}
g_k''(s) = -\frac{\sigma'_k(s)}{s} \leq 
\left\{
\begin{array}{ccl}
k (1+\sigma_0)  & \text{ for } & \displaystyle{s < \frac{k-1}{k^2}}, \\[8pt]
\displaystyle{ \frac{k^2 (2+\sigma_0)}{k-1}}  & \text{ for } & \displaystyle{s\ge \frac{k-1}{k^2}},
\end{array}
\right.
\end{equation}
and we obtain (A6) with $r=0$ (so that it also holds true for arbitrary $r \in (0,2)$).
\end{proof}

The previous proposition and {\bf Theorem~\ref{P.2}} ensure that, for any $T>0$ and $k \ge 4,$ there exists at least a  nonnegative weak solution $(h_k,\Gamma_k)$ to \eqref{a1}--\eqref{a4} with surface tension $\sigma_k$ and initial condition $(h_0,\Gamma_0).$ We prepare the study of compactness properties of the sequence $(h_k,\Gamma_k)_{k\ge 4}$ by deriving technical properties of the approximate surface tensions $(\sigma_k)_{k\ge 4}.$

%%%%%%%%%%%%%%%%%%%%%%%%%%%%%%%%%%%%%%%%%%%%%%%%%%%%%%%%%%%%%%%%
\begin{proposition}\label{pro.tec2}
If $\theta$ is the exponent given by \eqref{case.2}, then there exist constants $C_1,C_2 \in (0,\infty)$ such that, for $k\ge 4$,  
\begin{equation} \label{hyp.s.0}
0 \le g_k(s) \le C_1 \left( 1+s^2 \right) \quad  \mbox{ for all }\quad s\ge 0\,,
\end{equation}
\begin{equation} \label{hyp.s}
- (2+\sigma_0) \le \sigma'_k(s) \leq - C_2\, \dfrac{ ks}{(1+s)^{\theta}\left( 1 + k s\right)} \quad \mbox{ for all }\quad s \ge 0\,.
\end{equation}
Moreover, $(\sigma_k)_{k\ge 4}$ converges uniformly to $\sigma$ on compact subsets of $[0,\infty).$
\end{proposition}
%%%%%%%%%%%%%%%%%%%%%%%%%%%%%%%%%%%%%%%%%%%%%%%%%%%%%%%%%%%%%%%%

\begin{proof}
The first inequality in \eqref{hyp.s} having already been proved in \eqref{volvic}, we concentrate on the second inequality and first establish a similar estimate for $\tilde{\sigma}'_k.$ As the surface tension $\sigma$ satisfies \eqref{case.2}, there holds 
\begin{eqnarray*}
\tilde{\sigma}'_k(s) & \leq & - \dfrac{{ \sigma_1}}{(1+ s^\theta)} \leq -\dfrac{2^{\theta-1} { \sigma_1}}{(1+s)^{\theta}} \quad \mbox{ for }\;\; s\in \left[ \frac{1}{k} , k \right]\,, \\
\tilde{\sigma}'_k(s) & \leq & - \dfrac{s}{k^{1+\theta}} \leq - \frac{1}{(1+s)^\theta} \quad \mbox{ for }\;\; s\ge k \,,
\end{eqnarray*}
whence
$$
\tilde{\sigma}'_k(s) \le - \dfrac{C\, ks}{(1+ks) (1+s)^\theta} \quad \mbox{ for }\;\; s\ge \frac{1}{k} 
$$
since $ks/(1+ks)\le 1$ for $s\ge 0$. Also, 
$$
\tilde{\sigma}'_k(s) \le -ks \le - \dfrac{ks}{(1+ks)(1+s)^\theta} \quad \mbox{ for }\;\; s\in \left[ 0 , \frac{1}{k} \right]\,. 
$$
Consequently, if $s\ge (k-1)/k^2$, we have $s-(1/k^2)\ge (k-2)/k^2$ and, as $k\ge 4$, 
$$
2s \ge s + \dfrac{1}{k} \ge s + \dfrac{1}{k^2} \ge s - \dfrac{1}{k^2} \ge \dfrac{s}{2}\,,
$$
we have
$$
\sigma'_k(s) \leq -\dfrac{C\, (k(s - \frac{1}{k^2}))}{(1 + k (\frac{1}{k^2} + s)) (1+ (\frac{1}{k^2} + s))^{\theta}} \leq -\dfrac{C_2\, ks}{(1 + ks) (1+s)^{\theta}} \quad \mbox{ for }\;\; s\ge \frac{k-1}{k^2}\,.
$$
Since $\sigma'_k(s) = \tilde{\sigma}'_k(s)$ for $s\le (k-1)/k^2$ by \eqref{fantasio}, we end up with
$$
\sigma'_k(s) \leq - \dfrac{C_2\, ks}{(1+s)^{\theta} (1+ks)}\quad \mbox{ for }\;\; s\ge 0\,,
$$
and thus obtain \eqref{hyp.s}. We next note that, given $R>0$ and $s\in [0,R]$, it follows from \eqref{case.2} that, for $k\ge R$, we have
\begin{eqnarray*}
\left| \tilde{\sigma}_k(s) - \sigma(s) \right| & = & \left| \int_{\min{\{s,1/k\}}}^{1/k} \left[ \left( k \sigma'\left( \frac{1}{k} \right) - k \right)\ r - \sigma'(r) \right]\ dr \right| \\ 
& \le & \left( \left| \sigma'\left( \frac{1}{k} \right) \right| + 1 \right)\ \frac{1}{2k} + \sigma\left( \frac{1}{k} \right) - \sigma\left( \min{\left\{ s , \frac{1}{k} \right\}} \right)  \\
& \le & \frac{1+\sigma_0}{2k} + \sigma(0) - \sigma\left( \frac{1}{k} \right) 
\end{eqnarray*}
and
$$
\left| \tilde{\sigma}_k'(s) \right| \le \frac{1+\sigma_0}{k}  \quad \mbox{ for }\;\; s\in \left[ -\frac{1}{k^2} , 0 \right]\,.
$$
Consequently, owing to the continuity of $\sigma$ in $[0,\infty)$ and the properties of the convolution, the sequences $(\tilde{\sigma}_k)_k$ and $(\sigma_k)_k$ converge uniformly  to $\sigma$ on compact subsets of $[0,\infty).$

Finally, integrating the bound \eqref{volvic} gives $g_k(s)\le (2+\sigma_0)\ (s\ln{s} - s +1) \le (2+\sigma_0) (1+s^2)$ for $s\ge 0$, whence \eqref{hyp.s.0}.
\end{proof}

%%%%%%%%%%%%%%%%%%%%%%%%%%%%%%%%%%%%%%%%%%%%%%%%%%%%%%%%%%%%%%%%
\subsection{Compactness}\label{sec:comp}
%%%%%%%%%%%%%%%%%%%%%%%%%%%%%%%%%%%%%%%%%%%%%%%%%%%%%%%%%%%%%%%%

Let $T>0$. The main difference here with the strategy employed in {\bf Section~\ref{sec:cvgce}} is that we no longer have an estimate on $(\Gamma_k)_k$ in $L_\infty(0,T;L_2(0,1))$ but only in $L_\infty(0,T;L_1(0,1))$, and this requires a different approach to the compactness issue for $(\Gamma_k)_k.$ Let us collect the estimates available for $(h_k,\Gamma_k)_k$ which result from \eqref{a1}--\eqref{a4}, \eqref{a2021}, and the nonnegativity of $g_k$:  
\begin{enumerate}
\item Conservation of matter: for  $t \in [0,T],$ there holds
\begin{equation} \label{est.l1}
\int_{0}^1 h_k(t,x) \text{d$x$} = \int_{0}^1 h_0(x) \text{d$x$}\,, \qquad \int_{0}^1 \Gamma_k(t,x) \text{d$x$} = \int_{0}^1 \Gamma_0(x) \text{d$x$} \,.
\end{equation}
\item Energy estimate: for $t \in [0,T],$ there holds
\begin{equation} \label{est.gamma}
\begin{split}
\dfrac{1}{2}\ \int_{0}^1 |\partial_x h_k(t,x)|^2\, \text{d$x$} & + \int_0^T \int_{0}^1 \left[ \frac{h_k^3(s,x)\, \mathbf{1}_{(0,\infty)}(h_k{ (s,x)})}{21}|\partial_x^3 h_k(s,x)|^2 + \dfrac{h_k(s,x)}{8} |\partial_x \sigma_k(\Gamma_k(s,x))|^2 \right] \text{d$x$d$s$} \\
& - D \int_0^T \int_{0}^1 \dfrac{\sigma'_k(\Gamma_k(s,x))}{\Gamma_k(s,x)} |\partial_x \Gamma_k(s,x)|^2 \text{d$x$d$s$} \leq 
\int_{0}^1  \left[ \dfrac{|\partial_x h_0(x) |^2}{2} + g_k(\Gamma_0(x)) \right]\text{d$x$} . 
\end{split}
\end{equation}
\end{enumerate}
Moreover, both $h_k$ and $\Gamma_k$ are nonnegative a.e. in $(0,T)\times (0,1),$ and $\|g_k(\Gamma_0)\|_1 \le C_1\ (1+\|\Gamma_0\|_2^2)$ by \eqref{hyp.s.0}. Consequently,  \eqref{est.l1} and \eqref{est.gamma}, together with the lower bound \eqref{hyp.s} on $-\sigma'_k$ and the Poincar\'e inequality yield:
\begin{itemize}
\item[{\bf (B.1)}] $(h_k)_{k}$ is bounded in $L_\infty(0,T;H^1(0,1))$ and $(\Gamma_k)_{k}$ is bounded in $L_\infty(0,T;L_1(0,1)).$
\item[{\bf (B.2)}] $(h_k^{3/2}\, \mathbf{1}_{(0,\infty)}(h_k)\, \partial_x^{3}h_k)_{k},$ $(\partial_x \Gamma_k / (1+\Gamma_k)^{(1+\theta)/2})_k$, and $(\sqrt{h_k}\partial_x \sigma_k(\Gamma_k))_k$ are bounded in $L_2((0,T) \times (0,1)).$
%\item[{\bf (B.3)}] ${ (}\partial_x \Gamma_k / {k^{\frac{1+{\theta}}{2}}}{ )}$ is bounded in $L_2((0,T) \times (0,1)).$
\end{itemize}

%In order to improve these bounds, we interpolate {\bf (B.3)} with {\bf (B.1)} { we obtain} a uniform constant $C_{4} < \infty$ for which 
%\begin{equation} \label{bound.4}
%\|\Gamma_k\|_{L_2(0,T;L^4(0,1))} \leq C_4 k^{\frac{3}{8}(1+{\theta})}, \quad \forall \, k \in \mathbb N.
%\end{equation}

We then infer from \eqref{a14} (with surface tension $\sigma_k$), {\bf (B.1)}, and the embedding of $H^1(0,1)$ in $L_\infty(0,1)$ that
\begin{equation}
(h_k)_k \;\mbox{ is bounded in }\;{ L_\infty ((0,T)\times (0,1))} \;\;\mbox{ and }\;\; (\partial_t h_k)_k \;\mbox{ is bounded in }\; L_2(0,T;(H^1(0,1))^*)\,. \label{evian}
\end{equation}
Next, we prove the following embedding:

%%%%%%%%%%%%%%%%%%%%%%%%%%%%%%%%%%%%%%%%%%%%%%%%%%%%%%%%%%%%%%%%
\begin{lemma} \label{lem.emb}
Let { $\Gamma$ be a nonnegative function in $L_1(0,1)$ such that $(1+\Gamma)^{(1-\theta)/2} \in H^1(0,1)$.} Then
there exists $C_{\theta} < \infty$ depending only on $\theta$  such that, after possibly redefining $\Gamma$ on a set of measure zero, $\Gamma \in \mathcal{C}^{0,(1-\theta)/2}([0,1])$ together with
$$
\|\Gamma\|_{\mathcal C^{0,(1-\theta)/2}([0,1])} \leq C_{\theta} 
\left[ 1 + \int_0^1 \Gamma(x) \text{d$x$}\right] \left[ 1 + \int_0^1 \dfrac{|\partial_x \Gamma(x)|^2}{ (1+\Gamma(x))^{(1+\theta)}}\text{d$x$}\right].
$$ 
\end{lemma}
%%%%%%%%%%%%%%%%%%%%%%%%%%%%%%%%%%%%%%%%%%%%%%%%%%%%%%%%%%%%%%%%

\begin{proof} { Set 
$$
G := \frac{4}{(1-\theta)^2} \| \partial_x (1+\Gamma)^{(1-\theta)/2} \|_2^2=\int_0^1 \dfrac{|\partial_x \Gamma(x)|^2}{(1+\Gamma(x))^{1+\theta}} \text{d$x$}< \infty .
$$}
We assume $\Gamma$ to be smooth for simplicity and focus on the distance $\sqrt{1 + \Gamma(x)} - \sqrt{1 + \Gamma(y)}$ for $0\le x\le y\le 1$. 
Then, by H\"older's inequality
\begin{eqnarray*}
\left| \sqrt{ 1 + \Gamma(x)} \right. & - & \left. \sqrt{{ 1} + \Gamma(y)} \right| \le \int_x^y \dfrac{|\partial_x \Gamma(z)|}{\sqrt{{ 1} + \Gamma(z)}} \text{d$z$} \le \left[ \int_{x}^y \dfrac{|\partial_x \Gamma(z)|^2}{ (1+\Gamma(z))^{1+\theta}} \text{d$z$}\right]^{1/2} \left[ \int_{x}^y (1+\Gamma(z))^{\theta} \text{d$z$}\right]^{1/2}\\[8pt]
& \le & { \sqrt{G}\ \left[ \int_x^y (1+\Gamma(z))\, \text{d$z$} \right]^{\theta/2}\ |y-x|^{(1-\theta)/2}} \le { \sqrt{G}\ \left( 1 + \|\Gamma\|_1 \right)^{\theta/2}\ |y-x|^{(1-\theta)/2}} \,.
\end{eqnarray*}
 Since
$$
\int_0^1 \sqrt{1+\Gamma(z)} \text{d$z$} \leq \left( 1 + \|\Gamma\|_1 \right)^{1/2}\,,
$$
integrating the above inequality with respect to $y$ over $(0,1)$ ensures that $\|\sqrt{1+\Gamma}\|_\infty\le \left( 1 + \|\Gamma\|_1 \right)^{1/2} + \sqrt{G}\ \left( 1 + \|\Gamma\|_1 \right)^{\theta/2}$, so that there exists $C_{\theta}$ depending on $\theta$ only such that 
$$
\|\sqrt{{ 1}+ \Gamma}\|_{\mathcal{C}^{ 0,(1-\theta)/2}([0,1])} \leq C_{\theta} 
\left[ 1 + \int_0^1 \Gamma(x) \text{d$x$}\right]^{1/2} \left[ 1 + \int_0^1 \dfrac{|\partial_x \Gamma(x)|^2}{ (1+\Gamma(x))^{1+\theta}}\text{d$x$}\right]^{1/2}.
$$
We conclude  using the classical trick $\Gamma = (\sqrt{{ 1} + \Gamma})^2 - { 1}.$ 
\end{proof}

{ Now, {\bf (B.1)}, {\bf (B.2)}, and {\bf Lemma~\ref{lem.emb}} yield that 
\begin{equation}
(\Gamma_k)_{k} \;\mbox{ is bounded in }\; L_\infty(0,T;L_1(0,1)) \cap L_1(0,T ; \mathcal{C}^{0, (1-\theta)/2}([0,1]))\,.
\label{alet}
\end{equation}
In particular, since $\|\Gamma_k\|_2^2 \le \|\Gamma_k\|_\infty\, \|\Gamma_k\|_1$, we have that
\begin{equation}
(\Gamma_k)_{k} \;\mbox{ is bounded in }\; L_2((0,T)\times (0,1))\,.
\label{luchon}
\end{equation}
Owing to \eqref{hyp.s}, a first consequence of \eqref{luchon} is that $(\sigma_k(\Gamma_k))_k$ is also bounded in $L_2((0,T)\times (0,1))$. Furthermore, it follows from \eqref{est.gamma}, \eqref{hyp.s}, and \eqref{luchon} that 
\begin{eqnarray*}
\int_0^T \int_0^1 \left| \partial_x \sigma_k(\Gamma_k) \right|^{4/3}\, \text{d$x$d$s$} & = &  \int_0^T \int_0^1 \left( \frac{- \sigma_k'(\Gamma_k)}{\Gamma_k} \right)^{2/3}\ \left| \partial_x\Gamma_k \right|^{4/3}\ \left( \Gamma_k\ |\sigma_k'(\Gamma_k) | \right)^{2/3}\, \text{d$x$d$s$} \\
& \le & \left( \int_0^T \int_0^1 \frac{- \sigma_k'(\Gamma_k)}{\Gamma_k}\ \left| \partial_x\Gamma_k \right|^2\, \text{d$x$d$s$} \right)^{2/3}\ \left( \int_0^T \int_0^1 \left( \Gamma_k\ |\sigma_k'(\Gamma_k) | \right)^2\, \text{d$x$d$s$}\right)^{1/3} \\
& \le & C(T)\ (2+\sigma_0)^2\ \left( \int_0^T \int_0^1 \Gamma_k^2\, \text{d$x$d$s$}\right)^{1/3} \le C(T)\,.
\end{eqnarray*} 
Consequently,
\begin{equation}
(\sigma_k(\Gamma_k))_k \;\mbox{ is bounded in }\; L_{4/3}(0,T;W_{4/3}^1(0,1))\,.
\label{contrexeville}
\end{equation}
Finally, \eqref{a15} (with surface tension $\sigma_k$), {\bf (B.1)}, {\bf (B.2)}, \eqref{evian}, and \eqref{alet} guarantee that}
\begin{equation}
(\partial_t \Gamma_k)_{ k} \;\mbox{ is bounded in }\; L_{ 1}(0,T;({ H^2_N(0,1)})^*)\,,
\label{vichy}
\end{equation}
{ where $H_N^2(0,1):=\{ w\in H^2(0,1)\ :\ \partial_x w(0)=\partial_x w(1)=0 \}$.
Hence, { owing to the compactness of the embeddings of $H^1(0,1)$ and $\mathcal{C}^{0,(1-\theta)/2}([0,1])$ in $\mathcal{C}([0,1])$ and the continuity of the embedding of $\mathcal{C}([0,1])$ in either $(H^1(0,1))^*$ or $(H_N^2(0,1))^*,$ we infer from {\bf (B.1)}, \eqref{evian}, \eqref{alet}, \eqref{vichy}, and \cite[Corollary~4]{Si87} that there are a subsequence of $(h_k,\Gamma_k)_k$ (not relabeled)} and functions $h$ and $\Gamma$ such that
\begin{equation} \label{conv.Gh}
h_k \to h \text{ in $\mathcal{C}([0,T]\times [0,1])$} , \qquad  \Gamma_k \to \Gamma \text{ in ${ L_1}(0,T;\mathcal{C}([0,1]))$}\,.
\end{equation}
In addition, $(\partial_x h_{k})_k$ being bounded in $L_\infty(0,T;L_2(0,1))$  by {\bf (B.1)} { and $(\partial_x \sigma_k(\Gamma_k))_k$ { being bounded} in $L_{4/3}((0,T)\times (0,1))$ by \eqref{contrexeville},} we have, up to an extraction of a subsequence { and for some function $\Sigma$},
\begin{equation} \label{conv.dxh}
\partial_x h_k \rightharpoonup \partial_x h  \quad \text{weakly-$\star$ in $L_\infty(0,T;L_2(0,1))$ and } { \partial_x \sigma_k(\Gamma_k) \rightharpoonup \Sigma  \quad \text{in $L_{4/3}((0,T)\times (0,1))$}}\,. 
\end{equation}
As a consequence of { {\bf (B.1)}, \eqref{luchon}, \eqref{conv.Gh}, and \eqref{conv.dxh},} we get that the { limits satisfy}
\begin{equation} \label{reg.h}
h \in L_\infty(0,T ; H^1(0,1)), \quad { h\ge 0\,, }\quad \Gamma \in L_\infty(0,T;L_1(0,1))\cap L_2((0,T) \times (0,1))\,, \quad { \Gamma\ge 0}\,.
\end{equation}
{ Finally, thanks to {\bf (B.1)}, \eqref{conv.dxh}, and \eqref{reg.h}, we have
$$
\int_0^T \int_0^1 |\Gamma_k(t,x) - \Gamma(t,x)|^2\, \text{d$x$d$t$} \le \sup_{s\in [0,T]}{\left\{ \|\Gamma_k(s)\|_1+\|\Gamma(s)\|_1 \right\}}\ \int_0^T \|\Gamma_k(t) - \Gamma(t)\|_\infty\, \text{d$t$} \mathop{\longrightarrow}_{k\to\infty} 0\,,
$$
so that we also have
\begin{equation}
\Gamma_k \to \Gamma \text{ in ${ L_2}((0,T)\times (0,1))$}\,. \label{conv.sigma2}
\end{equation}
}

%%%%%%%%%%%%%%%%%%%%%%%%%%%%%%%%%%%%%%%%%%%%%%%%%%%%%%%%%%%%%%%%
\subsection{ Identifying the limit system}\label{sec:ils}
%%%%%%%%%%%%%%%%%%%%%%%%%%%%%%%%%%%%%%%%%%%%%%%%%%%%%%%%%%%%%%%%

According to the uniform bounds { {\bf (B.1)}}, {\bf(B.2)}, { and \eqref{evian},} we first obtain that, up to an extraction of a subsequence { and for some function $\overline{\jmath}_1$ and $\overline{\jmath}_2$},
\begin{equation} \label{conv.0}
h_k^{3/2}\ { \mathbf{1}_{(0,\infty)}(h_k)}\ \partial_x^3 h_k \rightharpoonup \overline{\jmath}_1, \qquad h_k^{ 1/2}\ \partial_x \sigma_k(\Gamma_k) \rightharpoonup \overline{\jmath}_2, \qquad \text{in $L_2((0,T)\times (0,1)).$}
\end{equation}
{ Arguing as in \cite[Section~3]{BF90} and \cite[Section~3.4]{GW06}, we first deduce from {\bf (B.2)} and \eqref{conv.Gh} that $\partial_x^3 h$ belongs to $L_2(\mathcal{P}(\delta))$ for all $\delta>0$ where $\mathcal{P}(\delta) := \{ (t,x)\in (0,T)\times (0,1)\ :\ h(t,x)>\delta \}$ and $\overline{\jmath}_1 = h^{3/2} \partial_x^3 h$ in $ \{ (t,x)\in (0,T)\times (0,1)\ :\ h(t,x)>0 \}$. Combining this result with \eqref{conv.Gh} yields 
\begin{eqnarray*}
& & \lim_{k\to\infty} \int_0^T \int_0^1 h_k^3\, \mathbf{1}_{(0,\infty)}(h_k)\, \partial_x^3 h_k\ \partial_x \zeta \, \text{d$x$d$s$} = \int_0^T \int_0^1 h^3\, \mathbf{1}_{(0,\infty)}(h)\, \partial_x^3 h\  \partial_x \zeta \, \text{d$x$d$s$} \,, \\
& & \lim_{k\to\infty} \int_0^T \int_0^1 h_k^2\, \mathbf{1}_{(0,\infty)}(h_k)\, \Gamma_k\, \partial_x^3 h_k\ \partial_x \zeta \, \text{d$x$d$s$} = \int_0^T \int_0^1 h^2\, \mathbf{1}_{(0,\infty)}(h)\, \Gamma\, \partial_x^3 h\  \partial_x \zeta \, \text{d$x$d$s$} \,,
\end{eqnarray*}
for any $\zeta \in {\mathcal{C}}^{\infty}( [0,T] \times [0,1])$ such that $\zeta(T,x)= 0$ for all $x \in [0,1]$ and $\partial_x \zeta(t,x) = 0$ for all $(t,x) \in [0,T] \times \{0,1\}.$ We next} claim the strong convergence
\begin{equation} \label{conv.sigma}
\sigma_k(\Gamma_k) \longrightarrow \sigma(\Gamma) \quad \text{in ${ L_2((0,T)\times(0,1))}$}.
\end{equation}
{ Indeed, on the one hand, we readily infer from \eqref{hyp.s} and \eqref{conv.sigma2} that 
$$
\int_0^T \int_0^1 \left| \sigma_k(\Gamma_k) - \sigma_k(\Gamma) \right|^2\, \text{d$x$d$t$} \le (2+\sigma_0)^2\ \int_0^T \int_0^1 |\Gamma_k - \Gamma|^2\, \text{d$x$d$t$} \mathop{\longrightarrow}_{k\to\infty} 0\,.
$$ 
On the other hand, it follows from Proposition~\ref{pro.tec2} and \eqref{hyp.s} that $\sigma_k(\Gamma)\to\sigma(\Gamma)$ a.e. in $(0,T)\times (0,1)$ with $|\sigma_k(\Gamma)|\le (2+\sigma_0)\ (1+\Gamma)\in L_2((0,T)\times (0,1)),$ whence
$$
\lim_{k\to\infty} \int_0^T \int_0^1 \left| \sigma_k(\Gamma) - \sigma(\Gamma) \right|^2\, \text{d$x$d$t$} = 0
$$
by the Lebesgue dominated convergence theorem. Thus, \eqref{conv.sigma} holds true. 

Together { with \eqref{contrexeville}}, the convergence \eqref{conv.sigma} ensures that $\sigma(\Gamma)\in L_{4/3}(0,T;W_{4/3}^1(0,1))$ and $\Sigma=\partial_x\sigma(\Gamma)$ in \eqref{conv.dxh}. Next,} collecting \eqref{conv.Gh}, \eqref{conv.dxh}, and \eqref{conv.sigma} yields $\overline{\jmath}_2 = h^{ 1/2} \partial_x \sigma(\Gamma),$ so that $h^{ 1/2} \partial_x \sigma(\Gamma) \in L_2((0,T) \times (0,1)).$ { It is then straightforward to pass to the limit as $k\to\infty$ in the remaining terms in the weak formulation \eqref{a14}-\eqref{a15} for $(h_k,\Gamma_k)$ and conclude} that $(h,\Gamma)$ is a weak solution to \eqref{a1}--\eqref{a4} with surface tension $\sigma$ and initial data $(h_0,\Gamma_0).$ This completes the proof of {\bf Theorem~\ref{P.1}}.

%%%%%%%%%%%%%%%%%%%%%%%%%%%%%%%%%%%%%%%%%%%%%%%%%%%%%%%%%%%%%%%%

\begin{remark}\label{re:final}%
We shall point out that our strategy to prove {\bf Theorem~\ref{P.1}} by approximating the surface tension $\sigma \in \mathcal{C}([0,\infty)) \cap \mathcal{C}^1(0,\infty)$ by surface tensions $(\sigma_k)_{k\in\mathbb N}$ satisfying the assumptions of {\bf Theorem~\ref{P.2}} does not { yield} existence of nonnegative weak solutions for the limiting case $\theta = 1.$ More precisely, if $\theta = 1$, then the analogue of {\bf Lemma~\ref{lem.emb}} for the energy dissipation merely yields a control on $\sqrt{\Gamma}$ in the space of continuous functions (instead of a H\"older space as in the case $\theta \in [0, 1)$), and we thus lose compactness of the concentration of any bounded family of solutions in $L_2((0,T)\times (0,1))$. It seems that this threshold is of high importance. Indeed, provided $-\sigma'(\Gamma)$ is dominated by $1/(1+\Gamma)$ at infinity, a good choice of multiplier for \eqref{a1}--\eqref{a2} yields that the integral \eqref{eq.diss} (with $\theta=1$) measures the dissipation of energy for any small solution. 
\end{remark}

\section*{Acknowledgement}
We gratefully acknowledge the hospitality of the Institut f\"ur Angewandte Mathematik der Leibniz Universit\"at Hannover and the Institut de Math\'{e}matiques de Toulouse, Universit\'{e} Paul Sabatier, where part of this work was done.

%%%%%%%%%%%%%%%%%%%%%%%%%%%%%%%%%%%%%%%%%%%%%%%%%%%%%%%%%%%%%%%%
%
%%%%%%%%%%%%%%%%%%%%%%%%%%%%%%%%%%%%%%%%%%%%%%%%%%%%%%%%%%%%%%%%
%

%
\end{document}